\documentclass{amsart}

\usepackage{comment}
\usepackage{graphicx,xcolor}

\numberwithin{equation}{section}
\usepackage[initials,nobysame,shortalphabetic]{amsrefs}
\usepackage{amssymb,amsmath,mathtools}
\usepackage{hyperref}
\usepackage{enumitem}
\usepackage[utf8]{inputenc}
\usepackage{mathrsfs}
\mathtoolsset{showonlyrefs=true}

\usepackage{caption}
\setlist[enumerate,1]{label=\upshape{(\roman*)},ref=\roman*}
\setlist[enumerate,2]{label=\upshape{(\alph*)},ref=\alph*}

\allowdisplaybreaks[2]

\renewcommand{\S}{\mathbb{S}^1}
\newcommand{\R}{\mathbb{R}}
\newcommand{\N}{\mathbb{N}}

\newcommand{\IP}[2]{\left<#1,#2\right>}

\DeclareMathOperator{\cn}{cn}
\DeclareMathOperator{\dn}{dn}
\DeclareMathOperator{\sn}{sn}
\DeclareMathOperator{\am}{am}

\newtheorem{thm}{Theorem}[section]
\newtheorem{prop}[thm]{Proposition}
\newtheorem{lem}[thm]{Lemma}
\newtheorem{cor}[thm]{Corollary}
\theoremstyle{definition}
\newtheorem{prob}[thm]{Problem}
\theoremstyle{remark}
\newtheorem{rem}[thm]{Remark}

\usepackage{xcolor}
\usepackage{marginnote}

\newcommand{\Gchanged}[1]{%
  {\color{blue} #1}%
}%

\begin{document}

\title[Scale-critical curve diffusion flows]{Scale-critical curve diffusion flows}

\author[T.~Miura]{Tatsuya Miura}
\address[T.~Miura]{Department of Mathematics, Graduate School of Science, Kyoto University, Kitashirakawa Oiwake-cho, Sakyo-ku, Kyoto 606-8502, Japan}
\email{tatsuya.miura@math.kyoto-u.ac.jp}

\author[G.~Wheeler]{Glen Wheeler}
\address[G.~Wheeler]{School of Mathematics and Physics \\
University of Wollongong\\
Northfields Avenue\\
Wollongong, NSW, 2522, Australia}
\email{glenw@uow.edu.au}

\date{\today}
\keywords{Curve diffusion flow, geometric flow, asymptotic analysis, stability}
\subjclass[2020]{53E40, 35B35, 35B40, 35K55}

\begin{abstract}
We introduce and study a one-parameter family of curve diffusion flows with a scale-critical cubic curvature term for closed immersed planar curves. We first classify all closed stationary solutions, showing that they are precisely circles or a unique family of ``super-lemniscates''. We then analyse the dynamical stability of homothetic circles. Under a sharp spectral condition, we establish, by purely variational methods, that any small perturbation of an $\omega$-fold circle monotonically approaches the unit $\omega$-circle after rescaling, translation, and reparametrisation. As a corollary, we determine the sharp ranges of the parameter for the stability of an embedded circle, and of all $\omega$-circles. We also uncover a striking arithmetic structure in the stability landscape, where the stability of $\omega$-circles depends non-monotonically on $\omega$.
\end{abstract}

\maketitle

\section{Introduction}

Higher-order geometric flows have been studied extensively, but their analysis is often delicate because, unlike for second-order flows, there is in general no maximum principle available.
Energy methods therefore play a central role, typically exploiting either a gradient-flow structure or a Lyapunov functional that controls the evolution.
Among the earliest examples is the curve diffusion flow introduced by Mullins in 1957 \cite{10.1063/1.1722742}, which can be regarded as an $H^{-1}$-gradient flow for length.

The broad aim of the present paper is to gain insight into how variational methods work for higher-order geometric flows that do not necessarily admit either a gradient-flow structure or a canonical Lyapunov functional.
In particular, we introduce and study a family of fourth-order geometric flows that includes the curve diffusion flow, which we call the \emph{critical curve diffusion flow} \eqref{CCDF}.

Given a parameter $c\in\R$ and a smooth immersed closed planar curve $\gamma_0:\S\rightarrow\R^2$, we consider the evolution equation
\begin{equation}
\label{CCDF}
\tag{CCDF}
\partial_t\gamma = -\left( k_{ss} + ck^3 \right)N\,,\qquad \gamma(\cdot,0) =: \gamma_0,
\end{equation}
where $\gamma:\S\times[0,T)\rightarrow\R^2$ is a one-parameter family of smooth
immersed curves.
Here $s$ denotes the arclength parameter, $N$ is the inward-pointing unit normal, and the curvature scalar is defined by $k :=
\IP{\gamma_{ss}}{N}$.

The cubic term $k^3$ naturally arises as the scale-critical exponent.
Indeed, if we consider a solution to $\partial_t\gamma = -\left( k_{ss} + ck^\alpha \right)N$ with a general power $\alpha$, then the parabolically rescaled flow $\gamma^\rho(x,t):=\rho\gamma(x,\rho^{-4}t)$, where $\rho>0$, solves the equation $\partial_t\gamma^\rho = -\left( k^\rho_{ss} + c\rho^{3-\alpha}(k^\rho)^\alpha \right)N$.
Hence, the flow is scale-invariant (and thus geometric) if and only if $\alpha=3$; thus equation \eqref{CCDF} is singled out.

Equation \eqref{CCDF} forms a one-parameter family that includes important flows as special cases; not only the curve diffusion flow ($c=0$) but also the Chen flow ($c=-1$) and the free elastic flow ($c=\frac12$).
The case $c=\frac12$ has a special gradient-flow structure for which the bending energy $\frac12\int_\gamma k^2\,ds$ decreases along the flow \cite{DKS02,MW25}; therefore, it can also be regarded as the Willmore flow for translation-invariant (cylindrical) surfaces.
The case $c=-1$ is not a gradient flow but still decreases the length \cite{CWW23}, and in addition, this principle easily extends to all $c<0$ (see Section \ref{subsec:immortal_blowup}).
However, for $c>0$ with $c\neq\frac12$, we are not aware of any global Lyapunov functional.
A key point of the present work is that these seemingly different flows
can be analysed within a unified variational framework.

Standard theory for parabolic geometric flows yields that, for any smooth initial datum $\gamma_0$, the flow \eqref{CCDF} exists uniquely up to a maximal existence time $T\in(0,\infty]$.
Our interest lies in the precise global behaviour of the solutions.

\subsection{Classification of stationary solutions}

Our first main result is a full classification of stationary solutions to \eqref{CCDF}.
For the three exemplar flows, stationary solutions are already classified: when $c=0$, it is clear that the only closed stationary solutions are circles, possibly multiply-covered, whereas if $c=-1$ or $c=\frac12$, then there do not exist any closed stationary solutions \cite{CWW23,MW25}.
From this, one may expect that there are no closed stationary solutions for $c\ne0$.
However, this turns out not to be the case.

For $c=\frac29$, the lemniscate of Bernoulli is a nontrivial stationary solution.
In fact, for a countable family of parameters $c\in(0,\frac29)$, there exist closed stationary solutions with turning number zero, which we call \emph{super-lemniscates} (see Section \ref{sec:stationary} and Figure \ref{fig:stationary-family}).
We further prove that there are no other stationary solutions.

\begin{thm}\label{thm:intro_stationary}
There exists a stationary solution to \eqref{CCDF} if and only if either
\[
c=0, \qquad\text{or}\qquad c=c_j:=\frac{2}{(4j-1)^2} \quad \text{for some integer $j\geq1$.}
\]
If $c=0$, the only stationary solution is a round circle, possibly multiply covered.
If $c=c_j$, the stationary solution is a unique super-lemniscate, up to similar transformations, reparametrisations, and multiple coverings.
\end{thm}

\begin{figure}[htbp]
    \centering
    \includegraphics[width=\textwidth]{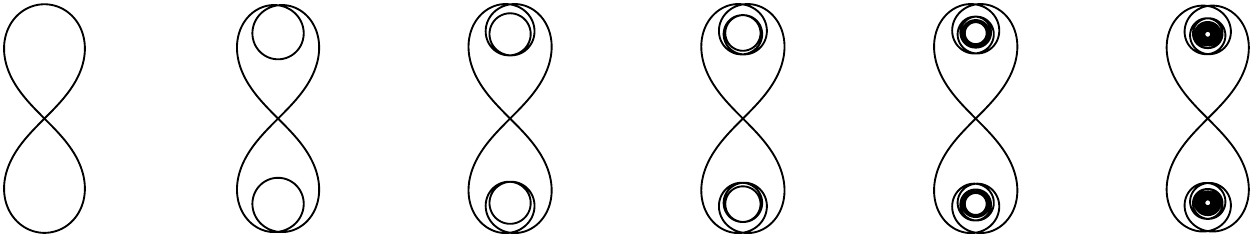}
    \caption{Super-lemniscates for
    $c_j=\frac{2}{(4j-1)^2}$, shown from left to right for $j=1,2,3,4,10,100$.
    The leftmost curve is the lemniscate of Bernoulli. }
    \label{fig:stationary-family}
\end{figure}

\begin{rem}
The sequence $\{c_j\}_{j=1}^\infty$ begins at the lemniscate of Bernoulli ($j=1$, $c_j=\frac29$) and, formally, ends at the circle ($j=\infty$, $c=0$).
The sequence of curves for $c_j\to0$ has two concentration points for the curvature and the arclength becomes unbounded, as may be observed from Figure \ref{fig:stationary-family}.
However, after rescaling each curve so that $\max k=1$, the curvature $k$ converges locally smoothly to $1$ as $j\to\infty$, thus yielding local smooth convergence to the universal cover of the unit circle.
\end{rem}

\subsection{Dynamical stability of circles}

Our second main result is concerned with the dynamical stability of $\omega$-fold round circles (we call $\omega$-circles), which arise as common homothetic solutions; more precisely, they are stationary for $c=0$, self-expanding for $c>0$, and self-shrinking for $c<0$.

It is classically known that for the curve diffusion flow ($c=0$) an embedded circle is geometrically stable \cite{MR1454678,EMS98,W13}, while multi-fold circles are unstable \cite{EMS98,MR1962022} (see also \cite{MO21}).
On the other hand, for the free elastic flow ($c=\frac12$), the situation changes and all self-expanding $\omega$-circles are stable under an appropriate rescaling \cite{MW25,AW25}.
For the Chen flow ($c=-1$), a self-shrinking embedded circle is also stable under rescaling \cite{CWW23}, but the multi-fold circles have not been correctly addressed in the literature (see Remark \ref{rem:Chen}).

Here we unify these dynamical stability results, taking some inspiration from \cite{AW25}, and develop a new purely variational framework that provides an almost complete description of the stability and instability of circles for the general flow \eqref{CCDF}.
As a result, we reveal nontrivial thresholds of the parameter $c$ for the stability of circles, while discovering an unexpected arithmetic stability pattern:
the stability of $\omega$-circles does not vary monotonically with $\omega$, instead exhibiting irregular alternations between stable and unstable modes.

For a closed immersed curve $\gamma$ with turning number $\omega\neq0$, which we assume $\omega>0$ throughout by reversing the parameter if necessary, we define the \emph{curvature oscillation} as the scale-invariant quantity
\[
K_{\mathrm{osc}}(\gamma):=
L(\gamma)\int_\gamma (k-\bar k)^2\,ds,
\qquad
\bar k:=\frac{1}{L(\gamma)}\int_\gamma k\,ds=\frac{2\pi\omega}{L(\gamma)}.
\]
We also define the polynomial
\begin{equation}\label{eq:def_p_c}
    p_c(x):=2x^4-(6c+2)x^2+8c,
\end{equation}
which arises as the Fourier symbol of a quadratic form that describes the linearised shape operator for $K_{\mathrm{osc}}$ at the unit $\omega$-circle along a \emph{length-normalised} version of the flow \eqref{CCDF}.
Removing the scaling and translation modes $x=0,\pm1$, set
\begin{equation}\label{eq:def_lambda_hat}
\hat\lambda_{c,\omega}
:=
\min_{n\in\mathbb Z\setminus\{0,\pm\omega\}}
p_c\Big(\frac{n}{\omega}\Big) \in \R.
\end{equation}
Then we obtain the following variational stability theorem for circles under a spectrally sharp condition.

\begin{thm}\label{thm:intro_stability}
    Let $c\in\R$ and $\omega\geq1$ be an integer.
    Assume
    \begin{equation}\label{eq:lambda>0}
        \hat\lambda_{c,\omega}>0.
    \end{equation}
    Then there exists $\varepsilon=\varepsilon(c,\omega)>0$ with the following property:
    
    Let $\gamma_0:\S\to\R^2$ be a closed planar curve with turning number $\omega>0$ such that
    \[
    K_{\mathrm{osc}}(\gamma_0)\le\varepsilon.
    \]
    Let $\gamma:\S\times[0,T)\to\R^2$ be the unique solution of \eqref{CCDF} with maximal existence time $T\in(0,\infty]$.
    Then $K_{\mathrm{osc}}(\gamma(\cdot,t))$ is non-increasing in $t\in[0,T)$ and converges to zero as $t\to T$.
    In addition, there exists a family of translations $p(t)\in\R^2$ such that, after reparametrisation, the flow satisfies the following convergence to the unit $\omega$-circle $\gamma_\omega$:
    \begin{enumerate}
    \item If $c>0$, then $T=\infty$ and
    \[
    \frac{\gamma(\cdot,t)-p(t)}{\sqrt[4]{4ct}}\to \gamma_\omega
    \qquad\text{smoothly as }t\to\infty.
    \]
    In particular, the solution $\gamma$ approaches an expanding circle.
    
    \item If $c=0$, then $T=\infty$ and, for some $\sigma_\infty\in(0,\infty)$,
    \[
    \gamma(\cdot,t)-p(t)\to \sigma_\infty\gamma_\omega
    \qquad\text{smoothly as }t\to\infty.
    \]
    In particular, the solution $\gamma$ approaches a stationary circle.
    
    \item If $c<0$, then $T<\infty$ and
    \[
    \frac{\gamma(\cdot,t)-p(t)}{\sqrt[4]{4|c|(T-t)}}\to \gamma_\omega
    \qquad\text{smoothly as }t\nearrow T.
    \]
    In particular, the solution $\gamma$ approaches a shrinking circle.
    \end{enumerate}
\end{thm}

\begin{rem}
    In the $c=0$ case, the curve diffusion flow preserves the signed area, so the limit scale $\sigma_\infty$ is uniquely and explicitly determined by the area of $\gamma_0$.
\end{rem}

Theorem \ref{thm:intro_stability} is sharp in the sense that it is complemented by the following variational instability result, except at the borderline case $\hat{\lambda}_{c,\omega}=0$.
This result seems to be new for all $c\in\R$.

\begin{thm}\label{thm:intro_instability}
    Let $c\in\R$ and $\omega\geq1$ be an integer.
    Assume
    \begin{equation}\label{eq:lambda<0}
        \hat\lambda_{c,\omega}<0.
    \end{equation}
    Then for any $\varepsilon>0$ there exists a solution $\gamma:\S\times[0,T)\to\R^2$ of \eqref{CCDF} with turning number $\omega$ such that $K_\mathrm{osc}(\gamma_0)\leq\varepsilon$ and such that $K_{\mathrm{osc}}(\gamma(\cdot,t))$ is strictly increasing on a time interval $[0,t_0)$ for some $t_0\in(0,T]$.
\end{thm}

\begin{rem}
    In fact, the above initial curve $\gamma_0$ can be taken even as an explicit smooth perturbation of the $\omega$-circle, as one can immediately see from the proof.
\end{rem}

As corollaries of the above results, we determine two sharp ranges of the parameter $c\in\R$ for the stability of circles.
In what follows, we say that an $\omega$-circle is \emph{stable} (resp.\ \emph{unstable}) if the assertion of Theorem \ref{thm:intro_stability} (resp.\ Theorem \ref{thm:intro_instability}) holds.

The first case is about the stability of an embedded circle.

\begin{cor}\label{cor:embedded_circle_stability}
    If $c<\frac32$, then an embedded circle is stable.
    In addition, if $c>\frac32$, then an embedded circle is unstable.
\end{cor}
    
In particular, this extends the known stability results for embedded circles for the special flows ($c=-1,0,\frac12$) to the sharp parameter range $c\in(-\infty,\frac32)$.
The threshold case $c=\frac32$ remains open.

The second case is about the stability of all $\omega$-circles.

\begin{cor}\label{cor:stability_all_omega}
    If $\frac19\leq c\leq 1$, then an $\omega$-circle is stable for all $\omega\geq1$.
    In addition, if either $c<\frac19$ or $c>1$, then an $\omega$-circle is unstable for some $\omega\geq1$.
\end{cor}

This extends the stability result for $\omega$-circles for the free elastic flow ($c=\frac12$) to the optimal parameter range $c\in[\frac19,1]$.

\begin{rem}\label{rem:Chen}
    In \cite{CWW23} it is claimed that all $\omega$-circles are stable for the Chen flow ($c=-1$), which contradicts our result.
    Unfortunately, the arguments in \cite{CWW23} contain some errors, and essentially work only for $\omega=1$.
    In particular, the Poincar\'e-type estimates used throughout in \cite{CWW23} mistakenly involve $\omega$ in the coefficients, yielding incorrect assertions for $\omega\geq2$.
    Indeed, if one considers a smooth perturbation of an $\omega$-circle as in the proof of Theorem \ref{thm:intro_instability} below, then we can construct an explicit counterexample to the key estimate \cite[Proposition 8]{CWW23} when $\omega\geq2$.
    Therefore, our result not only improves the $\omega=1$ case of \cite{CWW23} but also provides the first appropriate stability analysis for $\omega\geq2$.
\end{rem}

In fact, we obtain a more general criterion (Proposition \ref{prop:stability_interval}), which includes the above two corollaries as special cases. 
Moreover, this general criterion also yields the following observation.

\begin{rem}[Arithmetic stability pattern]\label{rem:arithmetic_stability_pattern}
    Numerical plots in Figure \ref{fig:stability_region_1} suggest that, for each $\omega\geq2$, the stability region with respect to the parameter $c$ is largely concentrated around the interval $[\frac19,1]$, while the exact regions exhibit irregular oscillations as $\omega$ varies.
    This behaviour reflects an underlying arithmetic structure, since stability is determined by the sign of $p_c(n/\omega)$.
    In particular, at the threshold values $c=1$ and $c=\frac19$, we have
    \[
    p_1(x)=2(x^2-2)^2,
    \qquad
    p_{\frac19}(x)=\frac{2}{9}(3x^2-2)^2,
    \]
    so that the transition between stability and instability is governed by how well rational numbers of the form $n/\omega$ approximate $\sqrt{2}$ and $\sqrt{2/3}$.
    As a consequence, for a fixed parameter $c\in\R$, the stability/instability transition in $\omega$ may not occur at a single threshold, but may instead involve irregular jumps, which is somewhat unexpected from a geometric viewpoint.
    For example, explicit computations and estimates imply that, for $c=1.001$, an $\omega$-circle is stable if and only if
    \[
    \omega\in\{1,2,3,4,6,8,9,11,13,16,18,23\},
    \]
    see also Figure \ref{fig:stability_region_2}.
    A similar phenomenon also occurs near $c=\frac19$.
\end{rem}

\begin{figure}[htbp]
    \centering
    \includegraphics[width=0.8\linewidth]{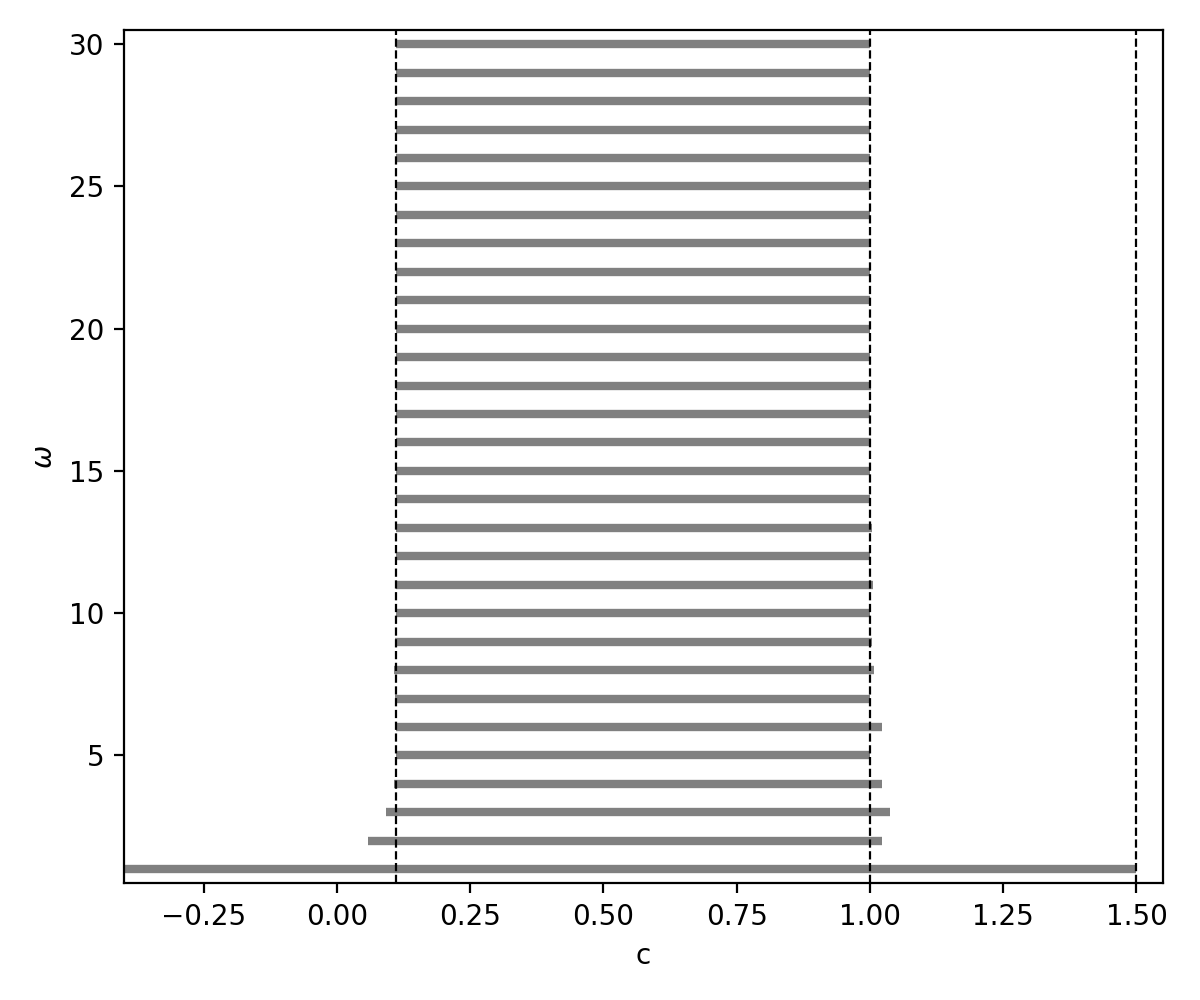}
    \caption{Stability region of the $\omega$-circle in the $(c,\omega)$-plane.
    For each integer $\omega\ge1$, the grey horizontal segment indicates the range of $c$ for which $\hat{\lambda}_{c,\omega}>0$, i.e., the $\omega$-circle is stable. The vertical dashed lines correspond to the thresholds $c=\frac19,1,\frac32$.}
    \label{fig:stability_region_1}
\end{figure}

\begin{figure}[htbp]
    \centering
    \includegraphics[width=0.8\linewidth]{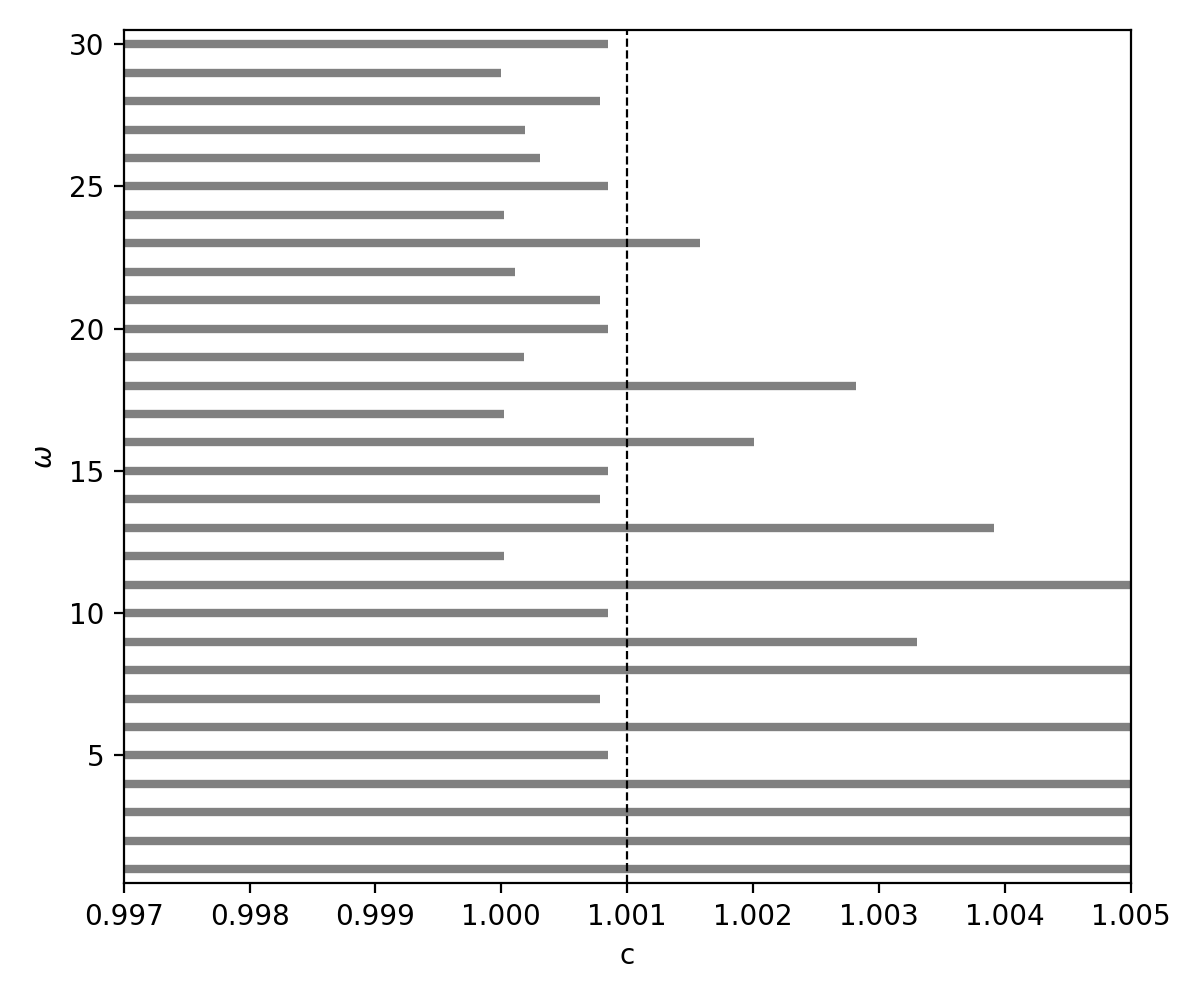}
    \caption{Zoom of the stability diagram around $c=1.001$.}
    \label{fig:stability_region_2}
\end{figure}

In the proof of Theorem \ref{thm:intro_stability}, the length is coupled to scale for $c\ne0$, so that we are able to separate shape from scale by passing to a length-preserving normalisation.
We then establish variational estimates to prove a general exponential convergence result for length-normalised flows \eqref{L-CCDF} under the same spectral condition (Theorem \ref{thm:normalised_circle_stability}), and finally translate it back to the original unnormalised flow.
The proof of Theorem \ref{thm:intro_instability} is based on an explicit construction of an unstable perturbation of an $\omega$-circle.
We emphasise that the argument is purely variational and does not rely on linear stability theory, despite the use of elementary Fourier expansions.

This paper is organised as follows:
In Section \ref{sec:stationary} we classify the stationary solutions, proving Theorem \ref{thm:intro_stationary}.
In Section \ref{sec:stability} we study the dynamical stability of circles and prove Theorem \ref{thm:intro_stability}, Theorem \ref{thm:intro_instability}, and their corollaries.
Finally, in Section \ref{sec:discussion} we discuss further properties and open problems.

\subsection*{Acknowledgements}
The first author is supported by JSPS KAKENHI Grant Numbers JP23H00085, JP23K20802, and JP24K00532.
The second author is partially supported by ARC grants 
FT250100880 and DP250101080.

\section{Classification of stationary solutions}\label{sec:stationary}

In this section we prove Theorem \ref{thm:intro_stationary}.
We first recall the definitions of elliptic functions and integrals (see also \cite[Appendix A]{Miura_elastica_survey} for details).
For $m\in(0,1)$ let
\[
F(x;m):=\int_0^x\frac{d\theta}{\sqrt{1-m\sin^2\theta}}, \quad K(m):=F(\pi/2;m).
\]
Note that $F'>0$ and hence invertible.
Let $\am(\cdot;m):=F^{-1}(\cdot;m)$ and
\begin{align}
    \cn(x;m)&:=\cos(\am(x;m)), \quad \sn(x;m):=\sin(\am(x;m)),\\
    \dn(x;m)&:=\sqrt{1-m\sn^2(x;m)}, \quad x\in\R.
\end{align}
The function $\cn$ (resp.\ $\sn$) is even (resp.\ odd) and anti-periodic with anti-period $2K(m)$, and in particular $4K(m)$-periodic.
Recall that
\[
\cn'=-\sn\dn, \quad \sn'=\cn\dn, \quad \dn'=-m\sn\cn.
\]

Now we define a family of closed curves which will be shown to be stationary solutions.
We call a curve $\gamma:\S\to\R^2$ a \emph{super-lemniscate} if the curve $\gamma$ has, up to rescaling and reparametrisation, the curvature 
\[
k(s)=\frac{1}{\sqrt{c}}\cn\Big(s;\frac{1}{2}\Big)
\]
with $c=c_j:=\frac{2}{(4j-1)^2}$ for some integer $j\geq1$.
The corresponding tangential angle $\theta$ is given by $\theta(s)=\sqrt{2/c}\arcsin( \sqrt{1/2}\sn(s;\frac{1}{2}) )$
and hence in particular, it ranges
\[
\left[-\sqrt{\frac{2}{c}}\frac{\pi}{4},\sqrt{\frac{2}{c}}\frac{\pi}{4}\right].
\]
On each quarter-period $\theta$ changes by exactly $\frac{4j-1}{4}\pi$, which causes the image of $\gamma$ to wind progressively more intensely as $j$ increases. This winding is matched by an equal un-winding in the next quarter-period, and thus the net turning of $\gamma$ is zero.

One of the main ingredients of the proof is the following classification for the corresponding ODE.

\begin{prop}\label{prop:ODE}
  For any $c\geq0$ and $k_0,k_1\in\R$ there is a unique solution $k:\R\to\R$ to the initial value problem
  \[
  k''+ck^3=0, \qquad k(0)=k_0,\quad k'(0)=k_1.
  \]
  If $c=0$, then $k(s)=k_1s+k_0$.
  If $c>0$, then
  $$k(s)=\frac{\alpha}{\sqrt{c}}\cn\Big(\alpha s+\beta;\frac{1}{2}\Big),$$
  where $k_0=\frac{\alpha}{\sqrt{c}}\cn(\beta;\frac{1}{2})$ and $k_1=-\frac{\alpha^2}{\sqrt{c}}\sn(\beta;\frac{1}{2})\dn(\beta;\frac{1}{2})$.
\end{prop}
\begin{proof}
  Uniqueness follows by standard ODE theory.
  It is also easy to directly check that the given function is a solution to the initial value problem.
  It remains to show that, in the case $c>0$, for any $k_0,k_1$ the desired constants $\alpha,\beta$ can be chosen.
  Up to changing the sign we may assume that $k_1\leq0$.
  If $k_0=0$ then we choose $\beta=K(\frac{1}{2})$ so that $\cn(\beta;\frac{1}{2})=0$ and $\sn(\beta;\frac{1}{2})\dn(\beta;\frac{1}{2})=1/\sqrt{2}$, and then choose $\alpha=\sqrt{-k_1\sqrt{2c}}$.
  Suppose that $k_0\neq0$.
  Let
  \[
  f(t):=\Gchanged{-}\sqrt{c}k_0^2\frac{\sn(t;\frac{1}{2})\dn(t;\frac{1}{2})}{\cn^2(t;\frac{1}{2})},
  \]
  which is continuous on the interval $(-K(\frac{1}{2}),K(\frac{1}{2}))$ and $\lim_{t\to-K(\frac{1}{2})^+}f(t)=+\infty$ and $\lim_{t\to K(\frac{1}{2})^-}f(t)=-\infty$.
  Then we can choose $\beta\in(-K(\frac{1}{2}),K(\frac{1}{2}))$ such that $f(\beta)=k_1$, and then choose $\alpha=\sqrt{c}k_0/\cn(\beta;\frac{1}{2})$.
\end{proof}

We also need the following elementary but key computational lemma.

\begin{lem}\label{lem:root}
  Let $f:\R\to\R$ be the smooth even function given by
  \[
  f(t):=\int_0^{\pi/2}\frac{\cos(tx)}{\sqrt{\cos{x}}}dx.
  \]
  Then $f(t)=0$ if and only if $|t|=\frac{4j-1}{2}$ for some integer $j>0$.
\end{lem}

\begin{proof}
  We first note that $f(t)>0$ whenever $|t|\leq1$ by the positivity of the integrand.

  Now we prove the recursive relation that
  \begin{equation*}\label{eq:induction}
    f(t)=-\frac{2t-3}{2t-1}f(t-2) \quad \text{for any}\ t>1.
  \end{equation*}
  Using $\cos{tx}=\cos{(t-1)x}\cos{x}-\sin{(t-1)x}\sin{x}$, we have
  \[
  \int_0^{\pi/2}\frac{\cos{tx}}{\sqrt{\cos{x}}}dx = \int_0^{\pi/2}\cos{(t-1)x}\sqrt{\cos{x}}dx - \int_0^{\pi/2}\frac{\sin{(t-1)x}\sin{x}}{\sqrt{\cos{x}}}dx.
  \]
  Integration by parts for the first term implies that, since $t\neq1$,
  \begin{align*}
    \int_0^{\pi/2}\frac{\cos{tx}}{\sqrt{\cos{x}}}dx &= \frac{1}{2(t-1)}\int_0^{\pi/2}\frac{\sin{(t-1)x}\sin{x}}{\sqrt{\cos{x}}}dx - \int_0^{\pi/2}\frac{\sin{(t-1)x}\sin{x}}{\sqrt{\cos{x}}}dx \\
    &= \frac{3-2t}{2(t-1)}\int_0^{\pi/2}\frac{\sin{(t-1)x}\sin{x}}{\sqrt{\cos{x}}}dx.
  \end{align*}
  Using $\sin{(t-1)x}\sin{x}=-\frac{1}{2}(\cos{tx}-\cos{(t-2)x})$, we get
  \[
  \int_0^{\pi/2}\frac{\cos{tx}}{\sqrt{\cos{x}}}dx = - \frac{3-2t}{4(t-1)}\int_0^{\pi/2}\frac{\cos{tx}}{\sqrt{\cos{x}}}dx + \frac{3-2t}{4(t-1)}\int_0^{\pi/2}\frac{\cos{(t-2)x}}{\sqrt{\cos{x}}}dx,
  \]
  and hence $\frac{2t-1}{4(t-1)}f(t)=\frac{2t-3}{4(t-1)}f(t-2)$.

  We finally complete the proof.
  If $1<t\leq2$, then by the above recursive relation and by the fact that $f(t-2)>0$, we deduce that $f(t)=0$ if and only if $t=\frac32$.
  Now let $Z_j:=\{t\in(2(j-1),2j]:f(t)=0\}$ for positive integers $j$.
  We already know that $Z_1=\{\frac{3}{2}\}$.
  On the other hand, by using the above recursive relation inductively in $j$, we deduce that $Z_j=\{\frac{4j-1}{2}\}$ for all $j\geq1$.
  Since $f$ is even, the proof is now complete.
\end{proof}

We are now ready to classify the closed stationary solutions.

\begin{proof}[Proof of Theorem \ref{thm:intro_stationary}]
  Throughout the proof we work with arclength parametrisation.
  The case $c=0$ is obvious.
  We argue the case $c<0$ by contradiction, so suppose that there is a closed stationary solution.
  Multiplying the equation $k_{ss}+ck^3=0$ by $k_s$ implies, after integration, that there is a constant $A\in\R$ such that for any $s\in\R/L\mathbb{Z}$ (where $L$ denotes the length),
  \[
  k_s^2(s)=\frac{-c}{2}k^4(s)-A.
  \]
  Since the curve is closed, there is a maximising point $s_M$ of $k^4$.
  Then $k_s(s_M)=0$, and hence $A=\frac{-c}{2}k^4(s_M)$.
  However this implies that the right-hand side is nonpositive and hence $k_s^2\equiv0$. 
  This contradicts that the curve is closed.

  We finally consider the case $c>0$.
  We first suppose that the equation $k''+ck^3=0$ admits a closed stationary solution and derive necessary conditions.
  By Proposition \ref{prop:ODE} and by the fact that $k$ is not identically zero, up to rescaling ($\alpha=1$) the curvature is necessarily represented by
  \[
  k(s)=\frac{1}{\sqrt{c}}\cn\Big(s+\beta;\frac{1}{2}\Big), \quad s\in\R/L\mathbb{Z},
  \]
  where $L$ denotes the length; for periodicity, $L$ needs to be of the form $L=4pK(\frac{1}{2})$ for some positive integer $p>0$.
  Up to shifting the parameter we may assume that $\beta=0$.
  Since the antiderivative of $\cn(x;m)$ is $\frac{1}{\sqrt{m}}\arcsin(\sqrt{m}\sn(x;m))$, the tangential angle $\theta$ (i.e., the antiderivative of $k$) must be represented by, up to rotation (i.e., up to addition of a constant to $\theta$),
  \[
  \theta(s)=\sqrt{\frac{2}{c}}\arcsin\left(\frac{1}{\sqrt{2}}\sn\Big(s;\frac{1}{2}\Big) \right).
  \]
  The unit tangent is $\partial_s\gamma(s)=(\cos\theta(s),\sin\theta(s))$, so the curve is closed if and only if $\int_0^L(\cos\theta(s),\sin\theta(s))\,ds=0$.
  Thanks to odd-symmetry and $2K(\frac{1}{2})$-anti-periodicity of $\sn(\cdot;\frac{1}{2})$, we have $\sin\theta(s)=-\sin\theta(L-s)$ and hence $\int_0^L\sin\theta(s)ds=0$ always holds.
  Therefore the remaining effective condition is that $\int_0^L\cos\theta(s)ds=0$.
  By symmetry and periodicity of $\cn(\cdot;\frac{1}{2})$, this is equivalent to the condition that
  \[
  \int_0^{K(\frac{1}{2})}\cos\theta(s)ds=\int_0^{K(\frac{1}{2})}\cos\left(\sqrt{\frac{2}{c}}\arcsin\left(\frac{1}{\sqrt{2}}\sn\Big(s;\frac{1}{2}\Big) \right)\right)ds=0.
  \]
  Through the changes of variables $x=\am(s,\frac{1}{2})$, $z=\frac{1}{\sqrt{2}}\sin x$, $\varphi=\arcsin{z}$, and $\varphi':=2\varphi$, we compute
  \begin{align*}
    &\int_0^{K(\frac{1}{2})}\cos\left(\sqrt{\frac{2}{c}}\arcsin\left(\frac{1}{\sqrt{2}}\sn\Big(s;\frac{1}{2}\Big) \right)\right)ds \\
    =& \int_0^{\pi/2}\cos\left(\sqrt{\frac{2}{c}}\arcsin\left(\frac{1}{\sqrt{2}}\sin{x} \right)\right)\frac{1}{\sqrt{1-\frac{1}{2}\sin^2x}}dx\\
    =& \int_0^{1/\sqrt{2}}\cos\left(\sqrt{\frac{2}{c}}\arcsin{z}\right)\frac{1}{\sqrt{1-z^2}\sqrt{1-2z^2}}dz\\
    =& \int_0^{\pi/4}\cos\left(\sqrt{\frac{2}{c}}\varphi\right)\frac{1}{\sqrt{1-2\sin^2\varphi}}d\varphi = \frac{1}{2}\int_0^{\pi/2}\cos\left(\frac{1}{\sqrt{2c}}\varphi'\right)\frac{1}{\sqrt{\cos{\varphi'}}}d\varphi'.
  \end{align*}
  By Lemma \ref{lem:root} the last integral vanishes if and only if $1/\sqrt{2c}=(4j-1)/2$, that is, $c=c_j$ for some integer $j\geq 1$.
  Therefore we obtain the necessity.

  The sufficiency is easy to verify by retracing the above computations in reverse.
  Thus the proof is complete.
\end{proof}

\section{Dynamical stability of circles}\label{sec:stability}

Throughout this section, let $\Gamma:\S\times[0,T)\to\R^2$ denote a solution of \eqref{CCDF}
with turning number $\omega>0$, and $\tau\in[0,T)$ denote its time variable.

\subsection{Length-preserving normalisation}

We first define a length-normalisation $\gamma(\cdot,t)$ of the solution $\Gamma(\cdot,\tau)$ to \eqref{CCDF}.
Fix the reference length
\[
L_0:=2\pi\omega.
\]
Let $\gamma_\omega$ denote the unit $\omega$-circle (of radius $1$ and centred at the origin) parametrised by arclength on $\R/(2\pi\omega\mathbb Z)$.

\begin{lem}\label{lem:length_preserving_normalisation}
Let
\[
\gamma(\cdot,t):=\sigma(\tau)^{-1}\Gamma(\cdot,\tau),
\quad
\sigma(\tau):=\frac{L(\Gamma(\cdot,\tau))}{L_0},
\quad
t=t(\tau):=\int_0^\tau\sigma(\tau')^{-4}\,d\tau'.
\]
Then the family of immersions $\gamma$ satisfies the length-normalised critical curve diffusion flow
\begin{equation}\label{L-CCDF}\tag{L-CCDF}
    \begin{split}
        \partial_t\gamma
        &=
        -(k_{ss}+ck^3)N+\lambda(t)\gamma,\\
        \lambda(t)&:=\frac{1}{L_0}\left(\int_\gamma k_s^2\,ds-c\int_\gamma k^4\,ds\right),
    \end{split}
\end{equation}
and
\[
L(\gamma(\cdot,t))\equiv L_0.
\]
In particular, the unit $\omega$-circle $\gamma_\omega$ is stationary for \eqref{L-CCDF}.
\end{lem}

\begin{proof}
Write the normal speed of \eqref{CCDF} as
\[
V[\Gamma]=-(k_{ss}+ck^3).
\]
Since $V[\rho\Gamma]=\rho^{-3}V[\Gamma]$ under dilations, differentiating
$\gamma=\sigma^{-1}\Gamma$ with respect to $\tau$ and then converting to the new time $t$ (using $dt/d\tau=\sigma^{-4}$) give
\[
\partial_t\gamma
=
-(k_{ss}+ck^3)N-\sigma^3\sigma'\gamma.
\]
Thus \eqref{L-CCDF} holds with
\[
\lambda(t)=-\sigma^3(\tau)\sigma'(\tau).
\]
On the other hand, since $\sigma=L(\Gamma)/L_0$,
\[
\frac{d}{d\tau}L(\Gamma)
=
-\int_\Gamma kV\,ds
=
-\int_\Gamma k_s^2\,ds+c\int_\Gamma k^4\,ds = \sigma^{-3}\left(-\int_\gamma k_s^2\,ds+c\int_\gamma k^4\,ds\right),
\]
which yields
\[
\lambda(t)
=
-\frac{\sigma^3}{L_0}\frac{d}{d\tau}L(\Gamma)
=
\frac{1}{L_0}\left(\int_\gamma k_s^2\,ds-c\int_\gamma k^4\,ds\right).
\]
Finally,
\begin{align}
    \frac{d}{dt}L(\gamma)
    &=
    -\int_\gamma k\Big(-(k_{ss}+ck^3)+\lambda\IP{\gamma}{N}\Big)\,ds
    \\
    &=
    -\int_\gamma k_s^2\,ds+c\int_\gamma k^4\,ds+\lambda\int_\gamma k\IP{\gamma}{N}\,ds.
\end{align}
Using $\int_\gamma k\IP{\gamma}{N}\,ds=-\int_\gamma |\gamma_s|^2\,ds=-L_0$ and the definition of $\lambda$
give $\frac{d}{dt}L\equiv0$.
\end{proof}

From now on we will mainly work with $\gamma$ and often omit the subscript $\gamma$ in the integral notation $\int_\gamma$.

\subsection{Evolution of curvature and its derivatives}

In this subsection we derive the curvature evolution equations for the length-normalised
flow \eqref{L-CCDF}.
The key point is that, although the normal velocity contains the
support function \(h=\IP{\gamma}{N}\), thanks to the special choice of the tangential gauge here, this contribution cancels pointwise in the
curvature evolution and in the commutator \([\partial_t,\partial_s]\).
As a consequence, the evolution of the higher derivatives of curvature has the same schematic form as in the standard Dziuk-Kuwert-Sch\"atzle framework \cite{DKS02}, with an additional zeroth-order term
involving \(\lambda\).

Let
\[
F:=k_{ss}+ck^3,
\qquad
h:=\IP{\gamma}{N},
\qquad T:=\partial_s\gamma, \qquad 
q:=\IP{\gamma}{T}.
\]
Then the length-normalised flow \eqref{L-CCDF} is
equivalent to
\begin{equation}\label{eq:normalised_equivalent_vector_form}
\partial_t\gamma=-FN+\lambda\gamma
=
(-F+\lambda h)N+\lambda qT.
\end{equation}
The flow then satisfies the following evolution laws.

\begin{lem}\label{lem:normalised_curvature_evolution}
Along the flow \eqref{L-CCDF}, the curvature satisfies
\begin{equation}\label{eq:k_t_normalised_compact}
k_t=-(F_{ss}+k^2F)-\lambda k 
=
-k_{ssss}
-(3c+1)k^2k_{ss}
-6ckk_s^2
-ck^5
-\lambda k.
\end{equation}
In addition,
\begin{equation}\label{eq:ds_t_normalised}
\partial_t ds=(\lambda+kF)\,ds,
\end{equation}
and hence the following commutator identity holds:
\begin{equation}\label{eq:commutator_normalised}
[\partial_t,\partial_s]=-(\lambda+kF)\partial_s.
\end{equation}
\end{lem}

\begin{proof}
Differentiating \(h=\IP{\gamma}{N}\) and \(q=\IP{\gamma}{T}\) in $s$ and using the Frenet-Serret formula $T_s=kN$ and $N_s=-kT$,
we find
\[
h_s=\IP{T}{N}+\IP{\gamma}{N_s}=-kq,
\qquad
q_s=\IP{T}{T}+\IP{\gamma}{T_s}=1+kh.
\]
Differentiating \(h_s=-kq\) once more gives
\begin{equation}
\label{eq:h_cancellation_identity}
h_{ss}=-k_sq-kq_s=-k_sq-k-k^2h\,.
\end{equation}
Now let
\[
V:=-F+\lambda h,
\qquad
\psi:=\lambda q.
\]
For a general velocity \(VN+\psi T\), one has the standard formulas
\[
k_t=V_{ss}+k^2V+\psi k_s,
\qquad
\partial_t ds=(\psi_s-kV)\,ds.
\]
Using \eqref{eq:h_cancellation_identity}, we compute
\begin{align*}
k_t
&=
(-F+\lambda h)_{ss}+k^2(-F+\lambda h)+\lambda qk_s
\\
&=
-(F_{ss}+k^2F)+\lambda(h_{ss}+k^2h+qk_s)
\\
&=
-(F_{ss}+k^2F)-\lambda k.
\end{align*}
This proves the first identity in \eqref{eq:k_t_normalised_compact}, and expanding \(F=k_{ss}+ck^3\) proves the second.

Finally,
\[
\psi_s-kV
=
\lambda q_s-k(-F+\lambda h)
=
\lambda(1+kh)+kF-\lambda kh
=
\lambda+kF,
\]
which proves \eqref{eq:ds_t_normalised}. 
The identity
\eqref{eq:commutator_normalised} is the standard consequence of \(\partial_t ds\).
\end{proof}

Set
\[
f:=k-1, \qquad
e(t):=\int f^2\,ds=\int k^2\,ds-2\pi\omega.
\]
Since $L\equiv2\pi\omega$ along the normalised flow, this $e$ is exactly
$K_{\mathrm{osc}}/(2\pi\omega)$.
A standard interpolation argument as in \cite{DKS02} yields the following criterion for immortality:

\begin{prop}\label{prop:immortal_criterion}
    Let $\gamma:\S\times[0,T)\to\R^2$ be a length-normalised flow \eqref{L-CCDF} with maximal existence time $T\in(0,\infty]$.
    If
    \[
    \sup_{t\in[0,T)}e(t)<\infty,
    \]
    then $T=\infty$.
\end{prop}

To prove this we introduce the usual \(P\)-notation for interpolation arguments.
For integers \(a,b,m\ge0\), we denote by
\[
P_b^{a,m}(k)
\]
a linear combination of terms of the form $\prod_{j=1}^b \partial_s^{i_j}k$ such that $i_j\le m$ for all $1\leq j\leq b$ and $\sum_{j=1}^b i_j=a$.
The coefficients in \(P_b^{a,m}(k)\) may depend on the fixed parameter $c\in\R$.

\begin{lem}\label{lem:normalised_higher_curvature_evolution}
For each integer \(m\ge0\), the $m$-th derivative of curvature satisfies 
\begin{equation}\label{eq:ksm_evolution}
\partial_t k_{s^m}
=
-k_{s^{m+4}}
-(m+1)\lambda\,k_{s^m}
+
P_3^{m+2,m+2}(k)
+
P_5^{m,m}(k).
\end{equation}
In addition,
\begin{align}\label{eq:ksm_energy}
\frac{d}{dt}\frac12&\int k_{s^m}^2\,ds
+
\int k_{s^{m+2}}^2\,ds
+
\frac{2m+1}{2}\lambda\int k_{s^m}^2\,ds
\\
\notag
&=
\int P_4^{2m+2,m+2}(k)\,ds
+
\int P_6^{2m,m}(k)\,ds.
\end{align}
\end{lem}

\begin{proof}
We argue by induction on \(m\).
For \(m=0\), \eqref{eq:ksm_evolution} is just \eqref{eq:k_t_normalised_compact}, namely
\[
k_t
=
-k_{ssss}
-(3c+1)k^2k_{ss}
-6ckk_s^2
-ck^5
-\lambda k,
\]
which has the required schematic form
\[
k_t=-k_{ssss}-\lambda k+P_3^{2,2}(k)+P_5^{0,0}(k).
\]

Assume now that \eqref{eq:ksm_evolution} holds for some \(m\ge0\). Then, using the
commutator formula \eqref{eq:commutator_normalised},
\begin{align*}
\partial_t k_{s^{m+1}}
&=
\partial_t\partial_s(k_{s^m})
=
\partial_s(\partial_t k_{s^m})+[\partial_t,\partial_s]k_{s^m}
\\
&=
\partial_s \big(
-k_{s^{m+4}}
-(m+1)\lambda k_{s^m}
+
P_3^{m+2,m+2}(k)
+
P_5^{m,m}(k)
\big)
\\
&\qquad -(\lambda+kF)k_{s^{m+1}}.
\end{align*}
Since \(\lambda\) depends only on \(t\), differentiating gives
\[
\partial_s\bigl(-(m+1)\lambda k_{s^m}\bigr)=-(m+1)\lambda k_{s^{m+1}}.
\]
Also,
\[
\partial_s P_3^{m+2,m+2}(k)=P_3^{m+3,m+3}(k),
\qquad
\partial_s P_5^{m,m}(k)=P_5^{m+1,m+1}(k).
\]
Finally,
\[
kF = k(k_{ss}+ck^3) = P_2^{2,2}(k)+P_4^{0,0}(k),
\]
and hence
\[
(kF)k_{s^{m+1}}
=
P_3^{m+3,\max\{2,m+1\}}(k)+P_5^{m+1,m+1}(k),
\]
which is absorbed into the same schematic classes as above. Therefore
\[
\partial_t k_{s^{m+1}}
=
-k_{s^{m+5}}
-(m+2)\lambda k_{s^{m+1}}
+
P_3^{m+3,m+3}(k)
+
P_5^{m+1,m+1}(k),
\]
which proves \eqref{eq:ksm_evolution} for \(m+1\).

We now derive \eqref{eq:ksm_energy}. Differentiate the energy:
\[
\frac{d}{dt}\frac12\int k_{s^m}^2\,ds
=
\int k_{s^m}\,\partial_t k_{s^m}\,ds
+
\frac12\int k_{s^m}^2\,\partial_t ds.
\]
Using \eqref{eq:ksm_evolution} and \eqref{eq:ds_t_normalised},
\begin{align*}
\frac{d}{dt}\frac12\int k_{s^m}^2\,ds
&=
-\int k_{s^m}k_{s^{m+4}}\,ds
-(m+1)\lambda\int k_{s^m}^2\,ds
\\
&\qquad
+\int k_{s^m}P_3^{m+2,m+2}(k)\,ds
+
\int k_{s^m}P_5^{m,m}(k)\,ds
\\
&\qquad
+\frac12\lambda\int k_{s^m}^2\,ds
+\frac12\int k_{s^m}^2\,kF\,ds.
\end{align*}
Integrating the first term of the right-hand side by parts twice, combining the two \(\lambda\)-terms, and observing that the remaining terms give the desired $P$-terms, we deduce \eqref{eq:ksm_energy}.
\end{proof}

\begin{lem}\label{lem:uniform_high_derivatives_normalised}
Let \(T\in(0,\infty]\), and let $\gamma:\S\times[0,T)\to\R^2$
be a length-normalised flow \eqref{L-CCDF}. Assume that
\[
\sup_{0\le t<T} e(t)\le E_0<\infty.
\]
Then, for every integer \(m\ge0\), there exists a constant
\[
C_m=C(m,c,\omega,E_0,\gamma_0)<\infty
\]
such that
\[
\sup_{0\le t<T}\int k_{s^m}^2\,ds\le C_m.
\]
\end{lem}

\begin{proof}
Set
\[
K_0:=\sup_{0\le t<T}\int k^2\,ds
=
2\pi\omega+\sup_{0\le t<T}e(t)
\le 2\pi\omega+E_0.
\]
We first obtain a lower bound on \(\lambda\). If \(c\le0\), then by definition
\[
\lambda(t)=\frac{1}{L_0}\left(\int k_s^2\,ds-c\int k^4\,ds\right)\ge0.
\]
If \(c>0\), then by the one-dimensional Gagliardo-Nirenberg inequality,
\[
\int k^4\,ds
\le
C\|k_s\|_2\|k\|_2^3 + C\|k\|_2^4
\le
\varepsilon \int k_s^2\,ds + C_{\varepsilon}(K_0).
\]
Choosing \(\varepsilon>0\) so small that \(1-c\varepsilon\ge\frac12\), we find
\[
\lambda(t)\ge -\Lambda_0
\qquad\text{for all }t\in[0,T),
\]
where \(\Lambda_0=\Lambda_0(c,\omega,E_0)\).

Fix an arbitrary \(m\ge0\). Let
\[
I_m(t):=\int k_{s^m}^2\,ds,
\qquad
J_m(t):=\int k_{s^{m+2}}^2\,ds.
\]
By Lemma \ref{lem:normalised_higher_curvature_evolution},
\begin{equation}\label{eq:energy_Im_Jm}
\frac12 I_m'(t)+J_m(t)+\frac{2m+1}{2}\lambda(t)I_m(t)
=
\int P_4^{2m+2,m+2}(k)\,ds
+
\int P_6^{2m,m}(k)\,ds.
\end{equation}
Using the lower bound on \(\lambda\), this gives
\begin{equation}\label{eq:energy_Im_Jm_rough}
\frac12 I_m'(t)+J_m(t)
\le
\left|\int P_4^{2m+2,m+2}(k)\,ds\right|
+
\left|\int P_6^{2m,m}(k)\,ds\right|
+
C I_m(t).
\end{equation}
Since $P_4^{2m+2,m+2}(k) = (k_{s^{m+1}} P_3^{m,m}(k))_s + P_4^{2m+2,m+1}(k)$, we have
\begin{equation}\label{eq:P-transform}
    \left|\int P_4^{2m+2,m+2}(k)\,ds\right|
\le \left|\int P_4^{2m+2,m+1}(k)\,ds\right|.
\end{equation}
We now apply the Dziuk-Kuwert-Sch\"atzle interpolation estimates
\cite[(2.16)]{DKS02}. Since \(\int k^2\,ds\le K_0\), for every \(\varepsilon>0\),
\begin{equation}\label{eq:DKS_lower_order_terms}
\left|\int P_4^{2m+2,m+1}(k)\,ds\right|
+
\left|\int P_6^{2m,m}(k)\,ds\right|
\le
\varepsilon J_m(t)+C_{\varepsilon,m},
\end{equation}
where \(C_{\varepsilon,m}=C_{\varepsilon,m}(c,\omega,E_0)\).

Similarly, again by the DKS interpolation inequalities and the uniform \(L^2\)-bound on \(k\),
for every \(\varepsilon>0\),
\begin{equation}\label{eq:DKS_Im_absorb}
I_m(t)\le \varepsilon J_m(t)+C_{\varepsilon,m}.
\end{equation}

Substituting \eqref{eq:DKS_lower_order_terms} with \eqref{eq:P-transform} into \eqref{eq:energy_Im_Jm_rough}, and then
using \eqref{eq:DKS_Im_absorb} to control the \(I_m\)-term, we obtain
\[
\frac12 I_m'(t)+J_m(t)
\le
\varepsilon J_m(t)+C_{\varepsilon,m}+C\bigl(\varepsilon J_m(t)+C_{\varepsilon,m}\bigr).
\]
Choosing \(\varepsilon>0\) sufficiently small, depending only on \(m,c,\omega,E_0\), yields
\begin{equation}\label{eq:Im_Jm_coercive}
I_m'(t)+c_m J_m(t)\le C_m
\qquad\text{for all }t\in[0,T).
\end{equation}
Now, combining \eqref{eq:DKS_Im_absorb} and \eqref{eq:Im_Jm_coercive}, we arrive at
\[
I_m'(t)+\hat c_m I_m(t)\le \hat C_m
\qquad\text{for all }t\in[0,T).
\]
Gr\"onwall's inequality gives $\sup_{0\le t<T} I_m(t)\le C_m$, completing the proof.
\end{proof}

\begin{proof}[Proof of Proposition \ref{prop:immortal_criterion}]
    Suppose on the contrary that $T_{\max}<\infty$.
    From the case \(m=1\) and the Gagliardo-Nirenberg inequality in the proof of Lemma \ref{lem:uniform_high_derivatives_normalised}, we obtain a
    uniform bound on \(\int k^4\,ds\), and hence \(\lambda\) is bounded on \([0,T_{\max})\). Since the
    curvature derivatives are uniformly bounded by Lemma \ref{lem:uniform_high_derivatives_normalised}, the quantity
    $F=k_{ss}+ck^3$
    is uniformly bounded on \([0,T)\) as well.
    Using the expression $\partial_t\gamma=-FN+\lambda\gamma$
    from Lemma \ref{lem:normalised_curvature_evolution}, we obtain
    \[
    |\gamma(\cdot,t)|
    \le |\gamma(\cdot,0)| + \int_0^t(|F|+|\lambda||\gamma|)\,dt' \leq C + C \int_0^t(1+|\gamma|)dt'.
    \]
    Gr\"onwall's inequality yields
    \[
    \sup_{0\le t<T_{\max}}\sup_{\S}|\gamma(\cdot,t)|<\infty.
    \]
    Thus both the curve and all curvature derivatives remain bounded on \([0,T_{\max})\).
    
    In addition, the parametrisation speed $v:=|\partial_x\gamma|$ is also well controlled so that $0<C^{-1}\leq v \leq C<\infty$ on the finite time interval $[0,T_{\max})$ since $\partial_tv=(\lambda+kF)v$ and $|\lambda+kF|$ is uniformly bounded.
    
    Therefore, as in \cite{DKS02}, standard continuation for fourth-order quasilinear geometric flows then implies a contradiction.
    Hence \(T_{\max}=\infty\).
\end{proof}

\subsection{Decay estimate for curvature oscillation}

Now we study the evolution of $e=K_{\mathrm{osc}}/(2\pi\omega)$, which is the main part of our analysis.
Hereafter we use the notation
\[
C=C(c,\omega)>0
\]
for a constant which depends only on $c,\omega$ and may change line by line. 



\begin{prop}\label{prop:quadratic_expansion_general_c}
Along the flow \eqref{L-CCDF},
\begin{equation}\label{eq:e_expansion_general_c}
\frac{d}{dt}e(t)
=
-Q_{c,\omega}[f]+R_{c,\omega}[f],
\end{equation}
where the quadratic term is given by
\begin{equation}\label{eq:quadratic_form}
    Q_{c,\omega}[f]
    :=
    2\int f_{ss}^2\,ds-(6c+2)\int f_s^2\,ds+8c\int f^2\,ds,
\end{equation}
and the remainder term is given by
\begin{align*}
R_{c,\omega}[f]
={}&
-(6c+3)\int f^2f_{ss}\,ds
-(2c+1)\int f^3f_{ss}\,ds
-16c\int f^3\,ds
-14c\int f^4\,ds
\\
&\quad
-6c\int f^5\,ds
-c\int f^6\,ds
-\frac{e}{2\pi\omega}\int f_s^2\,ds
+\frac{6c}{2\pi\omega}e^2
\\
&\quad
+\frac{4ce}{2\pi\omega}\int f^3\,ds
+\frac{ce}{2\pi\omega}\int f^4\,ds.
\end{align*}
Moreover, if $e(t)\le1$, then
\begin{equation}\label{eq:remainder_bound_general_c}
|R_{c,\omega}[f]|
\le
C\,\sqrt{e(t)}\left(\int f_{ss}^2\,ds+e(t)\right).
\end{equation}
\end{prop}

\begin{proof}
Using
\eqref{eq:k_t_normalised_compact} and \eqref{eq:ds_t_normalised}, we compute
\begin{align}
    \frac{d}{dt}e(t) &= \frac{d}{dt}\int k^2\,ds = \int 2k\big(-(F_{ss}+k^2F)-\lambda k\big)\,ds + \int k^2(\lambda+kF)\,ds\\
    &= -\int (2kF_{ss} + k^3 F) - \lambda\int k^2\,ds = -\int (2k_{ss} + k^3)F - \lambda(e+2\pi\omega). \label{eq:260317_01}
\end{align}
Since $L\equiv2\pi\omega$, we have $\int f\,ds=0$ and
\[
\lambda
=
-c+\frac{1}{2\pi\omega}\int f_s^2\,ds
-\frac{6c}{2\pi\omega}e
-\frac{4c}{2\pi\omega}\int f^3\,ds
-\frac{c}{2\pi\omega}\int f^4\,ds.
\]
Substituting $F=k_{ss}+ck^3$ and $k=1+f$ into \eqref{eq:260317_01}, expanding, and using $\int ff_{ss}\,ds=-\int f_s^2\,ds$ give
\eqref{eq:e_expansion_general_c} with the desired $Q_{c,\omega}[f]$ and $R_{c,\omega}[f]$.

To estimate the remainder $R_{c,\omega}[f]$ we use
\[
\|f_s\|_2^2=-\int ff_{ss}\,ds\le e^{1/2}\|f_{ss}\|_2,
\]
and the one-dimensional Gagliardo-Nirenberg inequality for $f$ with $\int f\,ds=0$:
\begin{equation}\label{eq:1D_GN}
    \|f\|_\infty
\le
C\big(e^{3/8}\|f_{ss}\|_2^{1/4}+e^{1/2}\big).
\end{equation}
These with Young's inequality imply
\[
\left|\int f^2f_{ss}\,ds\right|
=
2\left|\int ff_s^2\,ds\right|
\le
2\|f\|_\infty \|f_s\|_2^2
\le
C\sqrt e\big(\|f_{ss}\|_2^2+e\big),
\]
and similarly, also using $e\leq1$,
\[
\left|\int f^3f_{ss}\,ds\right|
=
3\left|\int f^2f_s^2\,ds\right|
\le
C\|f\|_\infty^2\|f_s\|_2^2 
\le
C\sqrt e\big(\|f_{ss}\|_2^2+e\big).
\]
For the pure-power terms with $m\in\{3,4,5,6\}$ we use
\[
\int |f|^m\,ds\le \|f\|_\infty^{m-2} e.
\]
By $e\le1$ and \eqref{eq:1D_GN}, repeated applications of Young's inequality give
\[
\int |f|^m\,ds
\le
C\sqrt e\big(\|f_{ss}\|_2^2+e\big),
\qquad m=3,4,5,6.
\]
The terms multiplied by $e$ are then of higher order, and \eqref{eq:remainder_bound_general_c} follows.
\end{proof}

Now we discuss more detailed estimates using
\[
f(s)=\sum_{n\in\mathbb Z}a_n e^{ins/\omega}, \qquad a_n:=\frac{1}{2\pi\omega}\int_0^{2\pi\omega}f(s)e^{-ins/\omega}ds,
\]
the Fourier series of $f$ with respect to the arclength parameter $s$ on $\R/(2\pi\omega)\mathbb Z$.
Then the Parseval identity yields
\[
e(t)=2\pi\omega\sum_{n\in\mathbb Z\setminus\{0\}}|a_n|^2,
\]
and also
\[
Q_{c,\omega}[f]
=
2\pi\omega\sum_{n\in\mathbb Z\setminus\{0\}}
p_c\Big(\frac{n}{\omega}\Big)\,|a_n|^2,
\qquad
p_c(x):=2x^4-(6c+2)x^2+8c.
\]
Recall that the polynomial $p_c$ is used in the definition of $\hat{\lambda}_{c,\omega}$, see \eqref{eq:def_lambda_hat}.
To relate $\hat{\lambda}_{c,\omega}$ with $Q_{c,\omega}$ more directly, we observe that the energy of the translation modes
\[
E_{\mathrm{tr}}(t):=2\pi\omega\big(|a_\omega(t)|^2+|a_{-\omega}(t)|^2\big)
\]
satisfies the following estimate.

\begin{lem}\label{lem:first_harmonics_general_c}
The following estimate holds:
\[
E_{\mathrm{tr}}(t) \le C e(t)^2.
\]
\end{lem}

\begin{proof}
Set
\[
I_1:=\int_0^{2\pi\omega} f(s)\cos s\,ds,
\qquad
I_2:=\int_0^{2\pi\omega} f(s)\sin s\,ds.
\]
After a rigid motion and a shift of the arclength origin, we may assume
\[
N(0)=(1,0)=(\cos0,\sin0).
\]
Let $R_\theta$ denote rotation by angle $\theta$, and let $J:=R_{\pi/2}$, so that
$N=JT$. Define
\[
w(s):=R_{-s}N(s).
\]
Differentiating and using $N_s=kJN$ gives $w_s=(k-1)Jw$ and hence $|w_s(s)|=|k(s)-1|$.
Therefore,
\[
|w(s)-(1,0)|
\le \int_0^s|w_s(\sigma)|\,d\sigma \leq
\int_0^{2\pi\omega}|k(\sigma)-1|\,d\sigma.
\]
Because $
N(s)=R_sw(s)$ and $(\cos s,\sin s)=R_s(1,0)$,
it follows that
\[
\sup_{s\in[0,2\pi\omega]}|N(s)-(\cos s,\sin s)|
\le
\int_0^{2\pi\omega}|k-1|\,ds.
\]
Now, since $\int (k-1)N\,ds
=
\int kN\,ds-\int N\,ds
=
0$,
we have
\begin{align*}
|I_1|^2+|I_2|^2
&=
\left|\int (k-1)(\cos s,\sin s)\,ds\right|^2
=
\left|\int (k-1)\Big((\cos s,\sin s)-N(s)\Big)\,ds\right|^2
\\
&\le
\left(\int |k-1|\,ds\right)^2
\sup_s |N(s)-(\cos s,\sin s)|^2
\le
(2\pi\omega)^2 e(t)^2.
\end{align*}

Finally, $\cos s=\frac12(e^{i\omega s/\omega}+e^{-i\omega s/\omega})$ and $\sin s=\frac1{2i}(e^{i\omega s/\omega}-e^{-i\omega s/\omega})$ yield
\[
I_1=\pi\omega(a_{-\omega}+a_{\omega}),
\qquad
I_2=\frac{\pi\omega}{i}(a_{-\omega}-a_\omega).
\]
Hence $|a_\omega|^2+|a_{-\omega}|^2
=
\frac{1}{2\pi^2\omega^2}(|I_1|^2+|I_2|^2
) \le
2 e(t)^2$,
yielding the desired estimate for $E_{\mathrm{tr}}$.
\end{proof}

Now we obtain a key spectral estimate for $Q_{c,\omega}$ in terms of $\hat{\lambda}_{c,\omega}$ in \eqref{eq:def_lambda_hat}.

\begin{lem}\label{lem:spectral_gap_general_c}
The following estimate holds:
\[
Q_{c,\omega}[f]
\ge
\hat\lambda_{c,\omega}e(t)-Ce(t)^2.
\]
\end{lem}

\begin{proof}
The Fourier expansions of $e$ and $Q_{c,\omega}$ and the definition of $\hat{\lambda}_{c,\omega}$ yield
\[
Q_{c,\omega}[f]
\ge
\hat\lambda_{c,\omega}(e-E_{\mathrm{tr}})+p_c(1)E_{\mathrm{tr}}
=
\hat\lambda_{c,\omega}e+(2c-\hat\lambda_{c,\omega})E_{\mathrm{tr}}.
\]
By Lemma \ref{lem:first_harmonics_general_c}, the desired estimate follows.
\end{proof}

Finally we prepare a simple differential inequality estimate.

\begin{lem}\label{lem:ode_general_c}
Let $y:[0,T)\to[0,\infty)$ be $C^1$ and suppose, for $\alpha>0$ and $A>0$,
\[
y'(t)\le -\alpha y(t)+A y(t)^{3/2}.
\]
Then there exists $\varepsilon_0=\varepsilon_0(\alpha,A)\in(0,\frac{1}{2}]$ such that
$y(0)\le\varepsilon_0$ implies
\[
y(t)\le 2y(0)e^{-\alpha t}
\qquad\text{for all }t\ge0.
\]
\end{lem}

\begin{proof}
Once $y(t_0)=0$ holds then $y(t)=0$ holds for all $t\geq t_0$, so we may assume $y>0$ on $[0,T)$ without loss of generality.
Set $v:=y^{-1/2}$.
Then
\[
v'=-\frac12 y^{-3/2}y'
\ge
\frac{\alpha}{2}v-\frac{A}{2}
\]
A comparison argument finishes the proof.
\end{proof}

We are now ready to prove the main decay estimate for the curvature oscillation.

\begin{thm}\label{thm:normalised_circle_stability}
Assume \(\hat\lambda_{c,\omega}>0\).
Then there exists \(\varepsilon=\varepsilon(c,\omega)>0\) such that if
\[
e(0)\le\varepsilon,
\]
then $e(t)$ is non-increasing along the flow \eqref{L-CCDF}, the maximal existence time is $T=\infty$, and
\[
e(t)\le 2e(0)e^{-\hat\lambda_{c,\omega}t}
\qquad\text{for all }t\ge0.
\]
\end{thm}

\begin{proof}
By the Parseval identity,
\[
\int f_{ss}^2\,ds
=
2\pi\omega\sum_{n\in\mathbb Z\setminus\{0,\pm\omega\}}
\Big(\frac{n}{\omega}\Big)^4|a_n|^2
+
E_{\mathrm{tr}}(t).
\]
Since \(\hat\lambda_{c,\omega}>0\), the polynomial \(p_c\) is strictly positive on the discrete set
\(\{n/\omega:n\in\mathbb Z\setminus\{0,\pm\omega\}\}\), and \(p_c(x)\sim 2x^4\) as \(|x|\to\infty\).
Hence there exists \(A=A(c,\omega)>0\) such that
\[
x^4\le A\,p_c(x)
\qquad\text{for all }x=\frac{n}{\omega},\quad n\in\mathbb Z\setminus\{0,\pm\omega\}.
\]
Therefore
\begin{align*}
\int f_{ss}^2\,ds
&\le
A\,2\pi\omega\sum_{n\neq0,\pm\omega}
p_c\Big(\frac{n}{\omega}\Big)|a_n|^2
+
E_{\mathrm{tr}}(t)
\\
&=
A\Big(Q_{c,\omega}[f]+p_c(1)E_{\mathrm{tr}}(t)\Big)+E_{\mathrm{tr}}(t).
\end{align*}
By Lemma \ref{lem:first_harmonics_general_c},
\begin{equation}\label{eq:H2_from_Q_general_c_new}
\int f_{ss}^2\,ds\le C\bigl(Q_{c,\omega}[f]+e(t)^2\bigr).
\end{equation}

Combining Proposition \ref{prop:quadratic_expansion_general_c} with
\eqref{eq:H2_from_Q_general_c_new} gives, whenever \(e(t)\le1\),
\begin{align*}
e'(t)
&\le
-Q_{c,\omega}[f]
+
C\sqrt{e(t)}\Bigl(\int f_{ss}^2\,ds+e(t)\Bigr)
\\
&\le
-Q_{c,\omega}[f]
+
C\sqrt{e(t)}\bigl(Q_{c,\omega}[f]+e(t)\bigr)
\\
&\le
-(1-C\sqrt{e(t)})Q_{c,\omega}[f]+Ce(t)^{3/2}.
\end{align*}
In particular, there exists $\varepsilon_1=\varepsilon_1(c,\omega)\in(0,1]$ such that if $e(t)\leq\varepsilon_1$ then the above prefactor $1-C\sqrt{e(t)}$ is positive.
Using Lemma \ref{lem:spectral_gap_general_c} and $\varepsilon_1\leq1$, we obtain
\begin{equation}\label{eq:e_ode_main}
e(t)\leq \varepsilon_1 \quad\Longrightarrow\quad e'(t)\le -\hat\lambda_{c,\omega}e(t)+C e(t)^{3/2}.
\end{equation}

Let \(T_{\max}\in(0,\infty]\) denote the maximal existence time of the normalised flow under consideration.
Set
\[
T_*:=\sup\{t<T_{\max}: e(s)\le\delta_* \text{ for all } s\in[0,t]\}, \quad \delta_*:=\min\Big\{\varepsilon_1,\frac{\hat{\lambda}_{c,\omega}^2}{4C^2}\Big\}.
\]
Then $0\leq e(t)\leq\delta_*\leq \varepsilon_1$ for $t\in [0,T_*)$ and hence, by \eqref{eq:e_ode_main}, $y:=e/\delta_*$ satisfies
\[
y'(t) \leq -\hat{\lambda}_{c,\omega}y(t) + C\sqrt{\delta_*}y(t)^{3/2} \qquad \text{for}\ t\in[0,T_*).
\]
Hence, applying Lemma \ref{lem:ode_general_c} to $y$, we deduce that
\[
e(0) \leq \delta_*\varepsilon_0(\hat{\lambda}_{c,\omega},C\sqrt{\delta_*}) \quad\Longrightarrow\quad e(t)\leq 2e(0)e^{-\hat{\lambda}_{c,\omega}t}.
\]
Therefore, if we suppose
\[
e(0)\le \varepsilon:=\min\Big\{\frac{\delta_*}{4},\delta_*\varepsilon_0(\hat{\lambda}_{c,\omega},C\sqrt{\delta_*})\Big\},
\]
then we obtain $e(t)\leq 2e(0)e^{-\hat{\lambda}_{c,\omega}t} \leq \delta_*/2$ for $t\in[0,T_*)$.
By continuity, we obtain $T_*=T$.
This together with Proposition \ref{prop:immortal_criterion} implies
\[
T_{\max}=\infty \quad\text{and}\quad e(t)\leq 2e(0)e^{-\hat{\lambda}_{c,\omega}t} \qquad \text{for all}\ t\geq0.
\]
Finally, we insert $e(t)\leq\delta_*\leq \hat{\lambda}_{c,\omega}^2/(4C^2)$ into \eqref{eq:e_ode_main} to deduce
\[
e'(t)\leq -\hat{\lambda}_{c,\omega}e(t)+Ce(t)\sqrt{\hat{\lambda}_{c,\omega}^2/(4C^2)} \leq -\frac{\hat{\lambda}_{c,\omega}}{2}e(t)\leq 0,
\]
which yields the desired non-increasing property.
\end{proof}

In particular, this result yields smooth exponential convergence.

\begin{cor}\label{cor:smooth_convergence}
    Assume \(\hat\lambda_{c,\omega}>0\), and let \(\varepsilon=\varepsilon(c,\omega)>0\) be the constant in Theorem \ref{thm:normalised_circle_stability}.
    If a solution $\gamma$ of \eqref{L-CCDF} satisfies $e(0)\leq\varepsilon$, then 
    \[
    \int k_{s^m}^2\,ds \leq C(m,c,\omega,\gamma_0)e^{-\hat{\lambda}_{c,\omega}t/2}.
    \]
    for all $t\geq0$ and all integers $m\geq1$.
    In particular, after translation and reparametrisation, the curve $\gamma(\cdot,t)$ converges to the unit $\omega$-circle smoothly and exponentially as $t\to\infty$.
\end{cor}

\begin{proof}
    Integrating by parts and using Theorem \ref{thm:normalised_circle_stability} and Lemma \ref{lem:uniform_high_derivatives_normalised} yield
    \[
    \int k_{s^m}^2\,ds = (-1)^m\int k_{s^{2m}}(k-1)\,ds \leq \|k_{s^{2m}}\|_{2}\sqrt{e(t)}\leq C(m,c,\omega,\gamma_0)e^{-\hat{\lambda}_{c,\omega}t/2}.
    \]
    Then, after arclength reparametrisation $s\in\R/(2\pi\omega\mathbb{Z})$, for every integer \(m\ge0\), we have $\|k(\cdot,t)-1\|_{H^m}\le C_m e^{-\hat{\lambda}_{c,\omega}t/2}$ and hence, by Sobolev embedding,
    \[
    \|k(\cdot,t)-1\|_{C^m}\le C_m e^{-\hat{\lambda}_{c,\omega}t/2}
    \qquad\text{for all }t\ge0.
    \]
    Since \(\theta_s=k\) and $\gamma_s=(\cos\theta,\sin\theta)$, after normalising so that $\gamma_s(0,t)=(\gamma_\omega)_s(0)$ by reparametrisation, we have the convergence for every integer $m\geq0$
    \[
    \| \gamma(\cdot,t)-a(t)-\gamma_\omega \|_{C^{m}} \le C_m e^{-\hat{\lambda}_{c,\omega}t/2} \qquad\text{for all }t\ge0,
    \]
    where $a(t)\in\R^2$ are translations.
    This proves the claim.
\end{proof}

\subsection{Convergence of the unnormalised flow}

Finally, we translate the above results back to the original unnormalised flow \eqref{CCDF}.
Since in this section we use $\gamma(\cdot,t)$ for the length-normalised flow \eqref{L-CCDF}, we denote by $\Gamma(\cdot,\tau)$ the corresponding unnormalised flow also in the proof of Theorem \ref{thm:intro_stability}.

We first determine the asymptotic rate of the rescaling factor $\sigma(t)=\sigma(\tau(t))$ in Lemma \ref{lem:length_preserving_normalisation}.

\begin{lem}\label{lem:lambda_asymptotics}
Assume \(\hat\lambda_{c,\omega}>0\), and let \(\varepsilon=\varepsilon(c,\omega)>0\) be the constant in Theorem \ref{thm:normalised_circle_stability}.
If a solution $\gamma$ of \eqref{L-CCDF} satisfies $e(0)\leq\varepsilon$, then
\[
|\lambda(t)+c|\le C(c,\omega,\gamma_0) e^{-\hat{\lambda}_{c,\omega}t/2}
\qquad\text{for all }t\ge0.
\]
In addition,
\[
\sigma(t)e^{-ct}\to \sigma_\infty\in(0,\infty)
\qquad\text{as }t\to\infty.
\]
\end{lem}

\begin{proof}
Applying Theorem \ref{thm:normalised_circle_stability} and Corollary \ref{cor:smooth_convergence} to the expression
\[
\lambda+c
=
\frac{1}{2\pi\omega}\left(
\int f_s^2\,ds
-6ce
-4c\int f^3\,ds
-c\int f^4\,ds
\right),
\]
we deduce that
\[
|\lambda(t)+c|\le C(c,\omega,\gamma_0) e^{-\hat{\lambda}_{c,\omega}t/2}.
\]
To discuss $\sigma(t)$, we recall from the proof of Lemma \ref{lem:length_preserving_normalisation} that $\lambda(t)=-\sigma^3(\tau)\sigma'(\tau)$ and $\frac{dt}{d\tau}=\sigma^{-4}$.
These yield $\frac{d}{dt}\sigma(\tau(t))=-\lambda(t)\sigma(\tau(t))$ and hence
\[
\frac{d}{dt}\log\big( e^{-ct}\sigma(\tau(t)) \big) = -(\lambda+c)
\]
so that
\[
e^{-ct}\sigma(t)
=
\sigma(0)\exp\left(-\int_0^t(\lambda(u)+c)\,du\right),
\]
and the integral converges absolutely thanks to the exponential decay of $|\lambda+c|$.
Hence $e^{-ct}\sigma(t)\to \sigma_\infty\in(0,\infty)$.
\end{proof}

Now we prove the convergence properties for the unnormalised flow $\Gamma$.

\begin{proof}[Proof of Theorem \ref{thm:intro_stability}]
    Let $\Gamma:\S\times[0,T)\to\R^2$ denote a solution to \eqref{CCDF} with time parameter $\tau$, and $\gamma$ be the corresponding length-normalised flow \eqref{L-CCDF} defined in Lemma \ref{lem:length_preserving_normalisation} with time parameter $t$.
    By Theorem \ref{thm:normalised_circle_stability} and the scale-invariance of $K_{\mathrm{osc}}$, we have the non-increasing and convergence properties of $K_{\mathrm{osc}}(\Gamma(\cdot,\tau))$ in $\tau$.
    
    It remains to show the trichotomy for the convergence of $\Gamma$.
    By Corollary \ref{cor:smooth_convergence} there is a family of translations $a(t)\in\R^2$ such that
    \[
    \gamma(\cdot,t)-a(t) \to \gamma_\omega
    \]
    smoothly as $t\to\infty$ after reparametrisation.
    Set $p(\tau):=\sigma(\tau)a(t(\tau))$.
    Then
    \begin{equation}\label{eq:Gamma-gamma_relation}
        \Gamma(\cdot,\tau)-p(\tau)
        =
        \sigma(\tau)\big( \gamma(\cdot,t(\tau)) - a(t(\tau)) \big).
    \end{equation}
    From Lemma \ref{lem:lambda_asymptotics},
    \begin{equation}\label{eq:sigma_infty}
        \sigma(t)=\sigma_\infty e^{ct}(1+o(1)) \qquad (t\to\infty).
    \end{equation}
    Using this with
    \begin{equation}\label{eq:dtau/dt}
        \frac{d\tau}{dt}=\sigma(t)^4
    \end{equation}
    from Lemma \ref{lem:length_preserving_normalisation},
    we obtain the three cases.
    
    If $c>0$, then l'H\^opital's rule with \eqref{eq:sigma_infty} and \eqref{eq:dtau/dt} yields
    \[
    \tau(t)\sim \frac{\sigma_\infty^4}{4c}e^{4ct},
    \quad\text{and hence}\quad
    \sigma(\tau)^4\sim 4c\tau.
    \]
    Therefore, by \eqref{eq:Gamma-gamma_relation} we have $(4c\tau)^{-1/4}\big(\Gamma(\cdot,\tau)-p(\tau)\big)\to\gamma_\omega$ smoothly as $\tau\to\infty$.
    
    If $c=0$, then \(\sigma(t)\to\sigma_\infty\), so $\Gamma(\cdot,\tau)-p(\tau) \to \sigma_\infty\gamma_\omega$ smoothly as $\tau\to\infty$.
    
    Finally, if $c<0$, then by \eqref{eq:sigma_infty}
    \[
    T:=\int_0^\infty \sigma(t)^4\,dt<\infty,
    \]
    and again l'H\^opital's rule with \eqref{eq:sigma_infty} and \eqref{eq:dtau/dt} yields
    \[
    T-\tau(t)\sim \frac{\sigma_\infty^4}{4|c|}e^{4ct},
    \quad\text{and hence}\quad
    \sigma(\tau)^4\sim 4|c|(T-\tau).
    \]
    By \eqref{eq:Gamma-gamma_relation} we have $(4|c|(T-\tau))^{-1/4}\big(\Gamma(\cdot,\tau)-p(\tau)\big)\to\gamma_\omega$ smoothly as $\tau\to T$.
\end{proof}

\begin{rem}
    From the above proof we can also obtain the convergence rates in the original time parameter $\tau$.
    Indeed, since $t(\tau)\sim\frac{1}{4c}\log\tau$ for $c>0$ and $t(\tau)\sim-\frac{1}{4|c|}\log(T-\tau)$ for $c<0$, as well as $t(\tau)\sim\sigma_\infty^{-4}\tau$ for $c=0$, we have
    \begin{align}
        K_\mathrm{osc}(\Gamma(\cdot,\tau)) \leq 
        \begin{cases}
         C(1+\tau)^{-\hat{\lambda}_{c,\omega}/(4c)} & (c>0),\\
         Ce^{-(\hat{\lambda}_{c,\omega}/\sigma_\infty^4)\tau} & (c=0),\\
         C(T-\tau)^{\hat{\lambda}_{c,\omega}/(4|c|)} & (c<0),
        \end{cases}
    \end{align}
    where $C=C(c,\omega,\gamma_0)>0$.
    In particular, when $c\leq0$, the condition $\hat{\lambda}_{c,\omega}>0$ holds if and only if $\omega=1$, where $\hat{\lambda}_{c,1}=24-16c\geq24$ (see also the proof of Corollary \ref{cor:embedded_circle_stability} below).
\end{rem}

\subsection{Instability of circles}

Here we prove the instability theorem by explicitly computing the curvature oscillation of a perturbed circle.

\begin{proof}[Proof of Theorem \ref{thm:intro_instability}]
Since the curvature oscillation is scale-invariant, it suffices to construct a suitable family of initial curves $\{\gamma_{0,\eta}\}_\eta$ such that $\gamma_{0,\eta}$ smoothly converges as $\eta\to0$ to the unit $\omega$-circle, and the length-normalised flow $\gamma_\eta$ of \eqref{L-CCDF} starting from $\gamma_{0,\eta}$ satisfies $\frac{d}{dt}K_\mathrm{osc}(\gamma_\eta(\cdot,0))>0$ for any small $\eta\neq0$.

By \eqref{eq:def_lambda_hat} and $p_c(x)\to\infty$ as $|x|\to\infty$, the minimum in the definition of $\hat\lambda_{c,\omega}$ is attained and negative.
As \(p_c\) is even, we may choose $n_0\in\N\setminus\{\omega\}$ such that
\[
p_c\left(\frac{n_0}{\omega}\right)=\hat\lambda_{c,\omega}<0.
\]

For \(\eta\in\mathbb R\), define the support function
\[
h_\eta(\vartheta):=1+\eta\cos\left(\frac{n_0\vartheta}{\omega}\right),
\qquad
\vartheta\in\mathbb R/(2\pi\omega\mathbb Z),
\]
and its radius of curvature
\[
\rho_\eta:=h_\eta+(h_{\eta})_{\vartheta\vartheta}
=
1+a\eta\cos\left(\frac{n_0\vartheta}{\omega}\right), \quad a:=1-\frac{n_0^2}{\omega^2}\neq 0.
\]
Hence \(\rho_\eta>0\) whenever \(|\eta|\) is sufficiently small, so that the support function $h_\eta$ defines a smooth closed locally convex curve $\gamma_{0,\eta}(\vartheta)=h_\eta(\vartheta)(\cos\vartheta,\sin\vartheta)+(h_\eta)_\vartheta(\vartheta)(-\sin\vartheta,\cos\vartheta)$ of turning number \(\omega\). 
Moreover,
\[
L(\gamma_{0,\eta})
=
\int_0^{2\pi\omega}\rho_\eta\,d\vartheta
=
\int_0^{2\pi\omega}h_\eta\,d\vartheta
=
2\pi\omega.
\]
Since $h_\eta\to1$ smoothly as $\eta\to0$, it is clear that $\gamma_{0,\eta}$ smoothly converges to the unit $\omega$-circle up to reparametrisation.
In particular, $K_{\mathrm{osc}}(\gamma_{0,\eta})\to0$ as $\eta\to0$.

Let \(\gamma_\eta(\cdot,t)\) be the 
normalised flow \eqref{L-CCDF} starting from \(\gamma_{0,\eta}\), and define
\[
e_\eta(t):=\int_{\gamma_\eta(\cdot,t)}(k_\eta-1)^2\,ds = \frac{K_{\mathrm{osc}}(\gamma_\eta(\cdot,t))}{2\pi\omega}.
\]
It remains to show that $e_\eta'(0)>0$ for all small $\eta\neq0$. 

By Proposition \ref{prop:quadratic_expansion_general_c},
\[
e_\eta'(0)=-Q_{c,\omega}[f_\eta]+R_{c,\omega}[f_\eta], \qquad f_\eta:=k_\eta-1=\rho_\eta^{-1}-1.
\]
Since
\[
f_\eta
=
-a\eta\cos\!\left(\frac{n_0\vartheta}{\omega}\right)+O(\eta^2)
\qquad\text{in }C^2\bigl(\mathbb R/(2\pi\omega\mathbb Z)\bigr),
\]
and
\[
ds=\rho_\eta\,d\vartheta=(1+O(\eta))\,d\vartheta,
\qquad
\partial_s=\rho_\eta^{-1}\partial_\vartheta=(1+O(\eta))\partial_\vartheta,
\]
we obtain
\[
(f_\eta)_s
=
a\eta\frac{n_0}{\omega}\sin\left(\frac{n_0\vartheta}{\omega}\right)+O(\eta^2),
\qquad
(f_\eta)_{ss}
=
a\eta\frac{n_0^2}{\omega^2}\cos\left(\frac{n_0\vartheta}{\omega}\right)+O(\eta^2).
\]
Therefore
\[
e_\eta(0)=\int f_\eta^2\,ds
=
\pi\omega a^2\eta^2+O(\eta^3),\qquad
\int (f_\eta)_s^2\,ds
=
\pi\omega a^2\Bigl(\frac{n_0}{\omega}\Bigr)^2\eta^2+O(\eta^3),
\]
and
\[
\int (f_\eta)_{ss}^2\,ds
=
\pi\omega a^2\Bigl(\frac{n_0}{\omega}\Bigr)^4\eta^2+O(\eta^3).
\]
Substituting these expansions into the definition of \(Q_{c,\omega}\), we find
\[
Q_{c,\omega}[f_\eta]
=
\pi\omega a^2
p_c\!\left(\frac{n_0}{\omega}\right)\eta^2
+O(\eta^3)
=
\pi\omega a^2\hat\lambda_{c,\omega}\eta^2
+O(\eta^3).
\]
On the other hand, by the remainder estimate in Proposition \ref{prop:quadratic_expansion_general_c},
\[
|R_{c,\omega}[f_\eta]|
\le
C\sqrt{e_\eta(0)}
\left(
\int (f_\eta)_{ss}^2\,ds+e_\eta(0)
\right)
=
O(|\eta|^3).
\]
Hence
\[
e_\eta'(0)
=
-\pi\omega a^2\hat\lambda_{c,\omega}\eta^2+O(|\eta|^3),
\]
which is positive for all sufficiently small \(\eta\neq 0\) due to the assumption \(\hat\lambda_{c,\omega}<0\).
The proof is now complete.
\end{proof} 

\subsection{Sharp parameter ranges for stability}

We now investigate the precise value of $\hat{\lambda}_{c,\omega}$.
In fact, both Corollary \ref{cor:embedded_circle_stability} and Corollary \ref{cor:stability_all_omega} directly follow from the following general criterion.

\begin{prop}\label{prop:stability_interval}
Let $\omega\ge1$ be an integer, and
\[
c_\omega^-:=
\max_{\substack{n\in\mathbb N\\ 1\leq n \leq \omega-1}}
\frac{n^2(\omega^2-n^2)}{\omega^2(4\omega^2-3n^2)},
\qquad
c_\omega^+:=
\min_{\substack{n\in\mathbb N\\ n>2\omega/\sqrt3}}
\frac{n^2(\omega^2-n^2)}{\omega^2(4\omega^2-3n^2)},
\]
with the convention $c_1^-:=-\infty$ (equivalently, $\max\varnothing=-\infty$).
Let $c\in\R$.
Then $\hat\lambda_{c,\omega}>0$ holds if and only if $c_\omega^- < c< c_\omega^+$.
In addition, $\hat\lambda_{c,\omega}<0$ holds if and only if either $c< c_\omega^-$ or $c>c_\omega^+$.
Furthermore, $c_1^+=\frac32$, and for all $\omega\geq2$,
\[
0< c_\omega^- < \frac19, \qquad 1< c_\omega^+,\qquad \sup_{\omega\geq2}c_\omega^-=\frac19, \qquad \inf_{\omega\geq2}c_\omega^+=1.
\]
\end{prop}

\begin{proof}
For $x\neq \pm 2/\sqrt3$, let
\[
g(x):=\frac{x^2(1-x^2)}{4-3x^2}.
\]
Since $p_c(x)=2(4-3x^2)\bigl(c-g(x)\bigr)$,
we have
\[
p_c(x)>0
\iff
\begin{cases}
c>g(x), & |x|<2/\sqrt3,\\[1mm]
c<g(x), & |x|>2/\sqrt3.
\end{cases}
\]
Therefore, by \eqref{eq:def_lambda_hat} the condition $\hat\lambda_{c,\omega}>0$ is equivalent to the simultaneous
inequalities
\[
c>\max_{\substack{n\in\mathbb N\setminus\{\omega\}\\ n<2\omega/\sqrt3}}
g\!\left(\frac n\omega\right),
\qquad
c<\min_{\substack{n\in\mathbb N\\ n>2\omega/\sqrt3}}
g\!\left(\frac n\omega\right),
\]
which is equivalent to the condition $c_\omega^-< c < c_\omega^+$; in particular, the left lower bound agrees with $c_\omega^-$ since $g>0$ on $(0,1)$ while $g<0$ on $(1,2/\sqrt{3})$. 
Negating the above consideration gives the characterisation
of $\hat\lambda_{c,\omega}<0$ similarly.

If $\omega=1$, the minimum of $c_1^+$ is clearly attained by $n=2$ since $g(n)$ is increasing for $n\geq2$; therefore, $c_1^+=\frac32$.

Now assume $\omega\geq2$.
The maximum in $c_\omega^-$ exists because the index set is finite, and the minimum in $c_\omega^+$ exists because $g(n/\omega)\to\infty$ as $n\to\infty$.
We have
\[
\max_{|x|<2/\sqrt3}g(x)=g(\pm\sqrt{2/3})=1/9, \qquad \inf_{|x|>2/\sqrt3}g(x)=g(\pm\sqrt{2})=1,
\]
both of which cannot be attained by rational numbers, so that
$c_\omega^-< 1/9$ and $1< c_\omega^+$.
In addition, considering sufficiently large $\omega$, we can pick rational numbers of the form $n/\omega$ which well approximate $2/\sqrt{3}$ and $\sqrt{2}$, yielding $\sup_{\omega\geq2}c_\omega^-=\frac19$ and $\inf_{\omega\geq2}c_\omega^+=1$.
Finally, since $g(1/\omega)>0$ we also have $c_\omega^->0$.
\end{proof}



Finally, we give a more precise formulation of Remark \ref{rem:arithmetic_stability_pattern}.
Rewriting the above criterion in terms of the roots $r_c^\pm$ of $p_c(x)$, we obtain the following equivalent criterion for $\omega\geq2$; the proof is omitted.

\begin{prop}
    The following properties hold:
    \begin{enumerate}
        \item If $c\leq0$, then $\hat{\lambda}_{c,\omega}<0$ holds for all integers $\omega\geq2$.
        \item Suppose that $c\in(0,\frac19)\cup(1,\infty)$ and $\omega\geq2$ is an integer, and let
        \[
         r_c^\pm:=\sqrt{\frac{3c+1\pm\sqrt{(c-1)(9c-1)}}{2}}.
        \]
        Then $\hat{\lambda}_{c,\omega}>0$ holds if and only if $\frac1\omega\N\cap[r_c^-,r_c^+]=\varnothing$.
        On the other hand, $\hat{\lambda}_{c,\omega}<0$ holds if and only if $\frac1\omega\N\cap(r_c^-,r_c^+)\neq\varnothing$.
    \end{enumerate}
\end{prop}

Therefore, an $\omega$-circle is unstable if the set $\frac1\omega\N$ intersects the interval $(r_c^-,r_c^+)$.
In particular, it is clear that, given any $c\not\in[\frac19,1]$, we have $\hat{\lambda}_{c,\omega}<0$ for all large integers $\omega$ such that $\frac{1}{\omega}<r_c^+-r_c^-$, and hence $\omega$-circles are stable at most for finitely many $\omega$.
The exact stability depends on the subtle relation between the rational numbers $\frac1\omega\N$ and the real numbers $r_c^-,r_c^+$.

\section{Discussion}\label{sec:discussion}

 We close this paper by discussing several additional properties and open problems regarding \eqref{CCDF}.

\subsection{Immortal solutions and finite-time blowups}\label{subsec:immortal_blowup}

We first address the maximal existence time along the flow \eqref{CCDF}.
If $c=0$, it is already known that both immortal solutions and finite-time blowup solutions exist (see \cite{MR1962022}; cf.\ \cite{miura2024asymptotic}).
If $c=\frac{1}{2}$, the special gradient-flow structure ensures that the solution is always immortal \cite{DKS02}.

When $c<0$ (including $c=-1$), it is easy to observe that the flow \eqref{CCDF} always decreases the length and develops a singularity in finite time.

\begin{prop}\label{prop:c<0_finite}
    Let $c<0$.
    Then any solution $\gamma:\S\times[0,T)\to\R^2$ to \eqref{CCDF} satisfies $\frac{d}{dt}L(\gamma)<0$ and has a maximal existence time $T\leq L(\gamma_0)^4/(64\pi^4|c|)<\infty$.
\end{prop}

\begin{proof}
    We compute
    \[
    \frac{d}{dt}L(\gamma)=\int k(k_{ss}+ck^3)\,ds=-\int k_s^2+c\int k^4\,ds < 0.
    \]
    In addition, since Fenchel's theorem and H\"older's inequality yield
    \[
    2\pi \leq \int|k|\,ds \leq L^\frac{3}{4} \left( \int k^4\,ds \right)^\frac{1}{4},
    \]
    we have
    \[
    \frac{d}{dt}L(\gamma) \leq -\frac{16\pi^4|c|}{L(\gamma)^3}.
    \]
    Hence $(L^4)'\leq-64\pi^4|c|$ so that $0\leq L(\gamma(\cdot,t))\leq \sqrt[4]{L(\gamma_0)^4-64\pi^4|c|t}$ for all $t\in[0,T)$.
    This yields the desired estimate for $T$.
\end{proof}

In the remaining case, we already know that at least there exist immortal solutions, e.g., expanding circles.
However, the following problem remains open.

\begin{prob}
    Let $c>0$ with $c\neq\frac12$.
    Does the flow \eqref{CCDF} exist for all time $t\geq0$ from any smooth initial curve $\gamma_0:\S\to\R^2$?
\end{prob}

This may be related to the following more abstract question.

\begin{prob}
    Let $c>0$ with $c\neq\frac12$.
    Does the flow \eqref{CCDF} admit a general Lyapunov functional?
\end{prob}

It would be particularly interesting to investigate, for $c>\frac32$, how the solution evolves after an embedded circle is deformed by an unstable mode.

\subsection{Self-similar solutions}

The stationary solutions are already classified in Theorem \ref{thm:intro_stationary}, and the stability of stationary circles for $c=0$ is already clarified.
Thus the natural direction is to investigate the super-lemniscates.

\begin{prob}
    Let $c=c_j=\frac{2}{(4j-1)^2}$ with an integer $j\geq1$.
    Then, is the stationary super-lemniscate stable along the flow \eqref{CCDF}?
    What about its multiple coverings?
\end{prob}

We then turn to homothetically self-similar solutions.
This is a particularly important class of solutions in the study of geometric flows, as they often arise as models of finite-time singularities or limits of immortal solutions.

A classical result of Abresch-Langer provides a full classification of all homothetic solutions of closed planar curves to the curve shortening flow \cite{MR845704} (see also \cite{MR2931330}).
On the other hand, it is well known that the classification of self-similar solutions to a higher-order flow such as \eqref{CCDF} is far more difficult.

However, the class of equations \eqref{CCDF} admits special homothetic solutions, which we can classify rather easily.
We first recall that a solution is \emph{homothetic} if, up to a translation, it is of the form $\rho(t)\gamma$ for a fixed profile $\gamma:\S\to\R^2$.
In particular, for a homothetic solution to \eqref{CCDF}, the corresponding profile curve $\gamma$ satisfies
\begin{equation}\label{eq:homothetic}
    k_{ss}+ck^3 = \Lambda h, \qquad h:=\langle \gamma,N\rangle,
\end{equation}
for some $\Lambda\in\R$.
Now we call $\gamma:\S\to\R^2$ a \emph{common homothetic solution} if for any $c\in\R$ there exists $\Lambda_c\in\R$ such that $\gamma$ solves \eqref{eq:homothetic} with $\Lambda=\Lambda_c$.

Clearly, circles are common homothetic solutions.
More nontrivially, the lemniscate of Bernoulli is also a common homothetic solution, thanks to the special property that the ratio of $k_{ss}$, $k^3$, and $h$ is constant along the curve \cite{EGMWW15} (see also \cite{CWW23,MW25}).

Building on Theorem \ref{thm:intro_stationary}, we can classify all common homothetic solutions.

\begin{prop}\label{prop:intro_common_homothetic}
    A curve $\gamma:\S\to\R^2$ is a common homothetic solution to \eqref{CCDF} if and only if it is either a round circle or the lemniscate of Bernoulli, up to similar transformations, reparametrisations, and multiple coverings.
\end{prop}

\begin{proof}
We only discuss the classification (the ``only if'' part). 
Taking $c=0$ and $c=1$ in \eqref{eq:homothetic}, we obtain $k_{ss}=\Lambda_0 h$ and $k_{ss}+k^3=\Lambda_1 h$.
Subtracting gives
\[
k^3=(\Lambda_1-\Lambda_0)h.
\]
If $k$ is constant, then $\gamma$ is a round circle and we are done. 
Hence we may assume that $k$ is non-constant. In particular, $\Lambda_1\neq\Lambda_0$, so there
is $A\neq0$ such that
\begin{equation}\label{eq:homothetic_01}
    h=A k^3.
\end{equation}
Substituting this into the $c=0$ equation yields, for some $\mu\in\mathbb R$,
\begin{equation}\label{eq:homothetic_02}
    k_{ss}=\mu k^3.
\end{equation}

Now we compute $\mu$.
Differentiating \eqref{eq:homothetic_01}, we get $3A k^2 k_s=-kq$, where $q:=\langle \gamma,T\rangle$ with $T:=\gamma_s$ and we used $h_s=-kq$.
Since $k$ is analytic with respect to the parameter $s$, we may divide the above equation by $k$ to get $q=-3Akk_s$.
Differentiating once more and using $q_s=1+kh$ and \eqref{eq:homothetic_01}, we infer that
\[
-3A(k_s^2+kk_{ss})=1+A k^4.
\]
On the other hand, by \eqref{eq:homothetic_02}, we have $k_s^2-\frac{\mu}{2}k^4=C$ for some constant $C$. 
Therefore
\[
-3AC-\Bigl(\frac{9}{2}\mu+1\Bigr)A k^4=1.
\]
Because $k$ is non-constant, the coefficient of $k^4$ must vanish.
Since $A\neq0$,
\[
\mu=-\frac29.
\]

Therefore, by \eqref{eq:homothetic_02} with $\mu=-2/9$ and by Theorem \ref{thm:intro_stationary}, the only possibility for $\gamma$ is the lemniscate of Bernoulli, up to invariances.
\end{proof}

As the stability of circles is already clarified, it is natural to ask the following:

\begin{prob}
    Let $c\in\R$.
    Then, is the homothetic lemniscate of Bernoulli stable along the flow \eqref{CCDF}?
    What about its multiple coverings?
\end{prob}

Finally, we mention that, although the circle and the lemniscate are so far the only known examples of homothetic solutions to equations of the form \eqref{CCDF}, a nontrivial family of closed planar homothetic solutions for the cases $c=0$ and $c=\frac12$ has recently been discovered in \cite{andrews2026jellyfish}.
The construction and partial classification of such nontrivial self-similar solutions to \eqref{CCDF} therefore remain interesting open problems.
See also the existence \cite{MR2916362,MR5002094} and stability \cite{MR4597906} of nontrivial self-similar solutions to the curve diffusion flow for non-compact complete curves.

\bibliography{CCDF}

\end{document}